\documentclass[12pt,a4paper,twoside]{amsart}

\usepackage{amsmath}
\usepackage{amssymb}
\usepackage{amscd,amsthm}
\usepackage{amsmath}
\usepackage{color}
\usepackage{amsthm}
\usepackage{amssymb}
\usepackage{amsfonts}
\usepackage{dsfont}
\usepackage[arrow, matrix, curve]{xy}
\usepackage{mathtools}      

\usepackage[pagebackref=true, colorlinks=false, citebordercolor={0.36 0.54 0.66}, 
linkbordercolor={0.74 0.83 0.9}]{hyperref} 

\newtheorem{prop}{Proposition}[section]
\newtheorem{theo}[prop]{Theorem}
\newtheorem{propo}[prop]{Proposition}
\newtheorem{rem}[prop]{Remark}
\newtheorem{coro}[prop]{Corollary}
\newtheorem{lem}[prop]{Lemma}

\DeclareMathOperator{\vol}{\mathrm{vol}}

\DeclareMathOperator{\Hom}{{\mathrm Hom}}
\DeclareMathOperator{\Pic}{\mathrm{Pic}}
\newcommand*{\defeq}{\coloneqq}

\newcommand*{{\sbullet}}{\scalebox{2}{$\cdot$}}

\def\Z{{\mathbb Z}}

\def\R{{\mathbb R}}
\def\Q{{\mathbb Q}}
\def\calt{\mathcal{T}}

\begin{document}
\title[Stringy Chern classes of singular toric varieties]
{Stringy Chern classes of singular toric varieties 
and their applications}

\author[Victor Batyrev]{Victor Batyrev}
\address{Mathematisches Institut, Universit\"at T\"ubingen, 
Auf der Morgenstelle 10, 72076 T\"ubingen, Germany}
\email{victor.batyrev@uni-tuebingen.de}

\author[Karin Schaller]{Karin Schaller}
\address{Mathematisches Institut, Universit\"at T\"ubingen, 
Auf der Morgenstelle 10, 72076 T\"ubingen, Germany}
\email{karin.schaller@uni-tuebingen.de}


\begin{abstract}
Let $X$ be a normal projective $\Q$-Gorenstein
variety with at worst log-terminal singularities. 
We prove a formula expressing  the total stringy Chern class 
of a generic complete intersection in $X$ via
the total stringy Chern class of $X$. 
This formula is motivated by its applications to mirror symmetry 
for Calabi-Yau complete intersections 
in toric varieties. 
We compute stringy Chern classes and give 
a combinatorial interpretation of the 
stringy Libgober-Wood identity for arbitrary 
projective $\Q$-Gorenstein toric varieties.
As an application we derive a new combinatorial identity 
relating $d$-dimensional reflexive polytopes 
to the number $12$ in dimension $d \geq 4$.
\end{abstract}

\maketitle

\thispagestyle{empty}
\section*{Introduction} 
The orbifold (or stringy) Euler number has been introduced by 
Dixon, Harvey, Vafa, and Witten \cite{DHVW85} as a new topological 
invariant of singular varieties motivated by string theory. It was observed 
by Hirzebruch and Hofer \cite{HH90} that the orbifold 
Euler number $e_{\rm orb}(X)$ of a singular Calabi-Yau variety 
$X$ can be identified
with the usual Euler number $e(Y)$ of its 
{\em Calabi-Yau desingularization} $Y$,
i.e., a desingularization $\rho: Y \to X$ 
such that $Y$ is a smooth Calabi-Yau manifold. 
Such a desingularization  $\rho$ might not be unique, but
the Euler number $e(Y)$ does not
depend on its choice. Moreover,  
the  orbifold (or stringy)  Hodge numbers $h^{p,q}_{\rm str}(X)$ of 
the singular Calabi-Yau variety $X$ can be defined in the same way 
as the Hodge 
numbers $h^{p,q}(Y)$ of the smooth Calabi-Yau manifold 
$Y$. 
The fact that the stringy Hodge numbers do not depend on the choice 
of a Calabi-Yau desingularization  
plays an important role in mirror symmetry for Calabi-Yau varieties 
\cite{BD96}.  

More generally, a desingularization 
$\rho: Y\to X$ of a normal projective variety 
$X$ with at worst canonical Gorenstein singularities is called {\em crepant}
if $\rho^* K_X = K_Y$. Using nonarchimedian integration \cite{Bat98}, 
one can prove  that the Euler number $e(Y)$ and the 
Hodge numbers
$h^{p,q}(Y)$ do not depend on the choice of the 
crepant desingularization
$\rho$. These numbers  are called 
{\em stringy Euler number} $e_{str}(X)$ respectively {\em stringy Hodge numbers}
$h^{p,q}_{str}(X)$ of the singular 
variety $X$. 

It was discovered by Allufi \cite{Alu04}
that not only the Euler 
number $e(Y)$ (the top Chern class $c_d(Y)$) of $Y$, but also  
the push-forwards of all other Chern classes $\rho_*c_k(Y) \in A^k(X)_\Q$ 
$(0 \leq k \leq d)$ are independent of a crepant desingularization  
$\rho : Y\to X$ of a $d$-dimensional variety $X$. 
This observation led to the notion of stringy Chern classes
$c_k^{str}(X) \in A^k(X)_\Q$ $(0 \leq k \leq d)$ of singular varieties $X$.
They have been introduced and  
developed in \cite{Alu05,dFLNU07}. We remark that the 
stringy Chern class $c_{d-1}^{str}(X)$ of a singular  $d$-dimensional 
projective variety $X$ appears in the stringy 
version of the Libgober-Wood formula \cite{Bat00}. 

In this paper, we are interested in  stringy Chern classes
of generic  complete intersections in toric varieties. Our interest
is motivated by the well-known construction of many examples
for Calabi-Yau varieties and their mirrors 
as hypersurfaces and complete intersections
in Gorenstein toric Fano varieties \cite{Bat94,BB97, BD96}. 
Another motivation for our paper was the search for a
higher dimensional generalization of the well-known combinatorial 
identities for reflexive polytopes of dimension $2$ and $3$. 
This generalization and its connection
to the Libgober-Wood formula for smooth manifolds \cite{LW90} 
(see also \cite{Sal96}) 
has been considered independently in the recent
preprint of  Godinho, von Heymann, and Sabatini \cite{GvHS16}. Independently 
some results of our paper were obtained in the 
preprint of Douai \cite{Dou16} motivated by Hertling's conjecture about
the variance of the spectrum of tame regular functions. 

\medskip

Let $X$ be a $d$-dimensional normal projective 
$\Q$-Gorenstein variety
with at worst {\em log-terminal
singularities}, i.e.,  
the canonical class $K_X$ of $X$ is a $\Q$-Cartier
divisor and for some 
desingularization $\rho : Y \rightarrow X$ 
of $X$, whose exceptional locus 
is a union of  smooth irreducible divisors $D_1, \ldots, D_s$
with only normal crossings,  one has  
\begin{align*}
K_Y=\rho^*  K_X  + \sum_{i=1}^s a_iD_i 
\end{align*}
for some rational numbers $a_i >-1$ $(1 \leq i \leq s)$. 
The above desingularization 
$\rho$ will be called {\em log-desingularization}
of $X$. 
For any nonempty subset $J \subseteq I \defeq\{1,\ldots,s\}$ we 
define  $D_J$ to be 
the subvariety $\cap_{j \in J} D_j$ together with
its  closed embedding $e_J: D_J \hookrightarrow Y$
and set $D_{\emptyset}\defeq Y$. 
We remark that 
the subvariety $D_J \subseteq Y$ is either empty 
or a smooth projective subvariety of $Y$ of codimension $|J|$. 

If $V$ is an arbitrary smooth 
$d$-dimensional projective variety, then 
the {\em $E$-polynomial} of $V$ is defined as
\begin{align*} 
E\left( V; u, v \right) \defeq \sum_{0 \leq p,q \leq d} 
(-1)^{p+q} h^{p,q}\left( V \right) u^{p} v^{q},
\end{align*}
where $h^{p,q} \left( V \right)$ are the Hodge numbers of $V$. 
Furthermore, the Euler 
number $e(V) = c_d(V)$ of $V$ equals $E\left(V;1,1 \right)$.
The {\em stringy $E$-function} of the singular variety 
 $X$ is a rational algebraic function in two 
variables $u, v$ defined by the formula 
\begin{align} \label{Estr}
E_{str}(X; u, v) \defeq \sum_{\emptyset \subseteq J \subseteq I}
E(D_J; u,v) \prod_{j \in J} \left( \frac{uv-1}{(uv)^{a_j +1} -1} -1 \right).
\end{align}
One can prove that the function $E_{str}(X; u, v)$ 
does not depend on the choice of the log-de\-sin\-gu\-la\-ri\-za\-tion $\rho$
\cite[Theorem 3.4]{Bat98}. As a special case, this formula implies 
$E_{str}(X;u,v) = E(Y;u,v)$
if $\rho$ is a crepant desingularization of $X$.
The {\em top stringy Chern class}  
(or the {\em stringy Euler number}) of $X$ is defined to be 
the limit of the stringy $E$-function \eqref{Estr}, i.e.,
\begin{align} \label{estr}
c^{str}_d(X) \defeq \lim_{u,v \to 1}  E_{str}(X; u, v) =
\sum_{\emptyset \subseteq J \subseteq I} c_{d-\lvert J \rvert}(D_J)
\prod_{j \in J} \left( \frac{-a_j}{a_j+1}\right),
\end{align}
where $ c_{d-\lvert J \rvert}(D_J) =e(D_J)$  denotes the Euler number 
of the smooth
subvariety $D_J \subseteq Y$. 
We regard the stringy top Chern class  $c^{str}_d(X)$ as a special case of 
the {\em $k$-th stringy Chern class} defined in  
\cite{Alu05,dFLNU07}  for any 
$k$ $(0 \leq k \leq d)$. 
In this paper, we apply the following formula 
for the computation of the { $k$-th 
stringy Chern class} of a 
$d$-dimensional normal 
projective $\Q$-Gorenstein variety $X$ 
with at worst log-terminal singularities 
using the  usual Chern classes of smooth projective 
subvarieties $D_J \subseteq Y$: 
\begin{align*}
c_{k}^{str} (X) \defeq \rho_* \bigg(
\sum_{\emptyset \subseteq J \subseteq I}
{e_J}_* c_{k-\lvert J \rvert}(D_J)
\prod_{j \in J} \left( \frac{-a_j}{a_j+1}\right) \! \! \bigg)  
\in A_{d-k}(X)_\Q.
\end{align*}
Here, $A_{d-k}(X)= A^k(X)$ denotes the 
Chow group of $(d-k)$-dimensional cycles on $X$ modulo 
rational equivalence and $A_{\sbullet}(X)_{\Q} \defeq \bigoplus_{k=0}^d A_{d-k}(X)_{\Q}$ 
with $A_{d-k}(X)_\Q \defeq A_{d-k}(X) 
\otimes \Q$ the rational Chow ring of $X$.
Moreover, $\rho_* : A_{d-k}(Y) \to A_{d-k}(X)$ and
${e_J}_* :A_{d-k}(D_J) \to A_{d-k}(Y)$ are push-forward
homomorphisms corresponding to the proper birational morphism 
$\rho : Y \to X$ respectively the closed embeddings 
$e_J : D_J \hookrightarrow Y$.
It is important to note that the above definition
of stringy Chern classes
is also independent of the 
log-desingularization $\rho$.

The paper is organized as follows:

In Section \ref{section1}, we prove that the 
well-known formula expressing  
the total Chern class of smooth
complete intersections via the total Chern class of 
the ambient smooth variety $V$  
remains valid also for the
total stringy Chern class $ c_{\sbullet}^{str}(\cdot)$ 
of generic hypersurfaces and complete intersections 
in the singular ambient variety $X$, i.e.,
\begin{align*} 
i_* c_{\sbullet}^{str} \left(Z_1 \cap \ldots \cap Z_r  \right) 
= c_{\sbullet}^{str}(X) .\prod_{j=1}^r\left[Z_j\right]( 1+\left[Z_j\right])^{-1},
\end{align*}
where $Z_1,\ldots, Z_r$ are generic {semiample}
Cartier divisors
on $X$ and  $i: Z_1 \cap \ldots \cap Z_r \hookrightarrow X $ 
is the corresponding closed embedding. 
In particular, we show that 
the top stringy Chern class (or stringy Euler number) of generic  
semiample Cartier divisors $Z$ on $X$ 
can be computed  via  stringy Chern classes of $X$ by 
\begin{align*} 
c_{d-1}^{str} (Z) =
e_{str}(Z) = [Z].c_{d-1}^{str} (X) - [Z]^2.c_{d-2}^{str} (X) + 
\ldots = 
\sum_{k=1}^d (-1)^{k-1} [Z]^k.c_{d-k}^{str} (X).
\end{align*}
We give a similar formula for the top stringy Chern class 
$c_{d-r}^{str} \left(Z_1 \cap \ldots \cap Z_r  \right)$ of 
complete intersections
$Z_1 \cap \ldots \cap Z_r $, where $Z_1,\ldots, Z_r$ 
are generic semiample Cartier divisors on the singular variety $X$.

In Section \ref{section2},  we look at intersection numbers of stringy Chern classes 
with $\Q$-Cartier divisors. A particular case of such an intersection number 
has appeared in the stringy version of the Libgober-Wood identity
\begin{align} \label{eq1}
\frac{d^2}{d u^2} E_{str}\left(X;u,1\right)\Big\vert_{u=1} =
\frac{3d^2-5d}{12} c_d^{str}(X) + \frac{1}{6} c_1(X).c_{d-1}^{str}(X) 
\end{align}
in \cite[Theorem 3.8]{Bat00}, where the 
intersection number $c_1(X).c_{d-1}^{str}(X)$
has been defined as
\begin{align}\label{eq12}
c_1(X).c_{d-1}^{str}(X)\defeq
\sum_{\emptyset \subseteq J \subseteq I} \rho^* c_1(X).{e_J}_*c_{d-\lvert J \rvert-1}(D_J)
\prod_{j \in J} \left( \frac{-a_j}{a_j+1}\right)
\end{align}
and its independence on the choice of the
log-desingularization $\rho$ has been shown in 
\cite[Corollary 3.9]{Bat00}. 
In this paper,  
we consider more general  intersection numbers 
$ [Z_1].\; \ldots \; .[Z_k].c_{d-k}^{str} (X)$, where $Z_1,\ldots,Z_k$ 
are arbitrary  $\Q$-Cartier divisors.  These intersection numbers 
can be defined by a similar formula
\begin{align*}
[Z_1].\; \ldots \; .[Z_k].c_{d-k}^{str} (X)  \defeq 
\sum_{\emptyset \subseteq J \subseteq I}
\rho^*[Z_1].\; \ldots \; .\rho^*[Z_k]. {e_J}_*c_{d-\lvert J \rvert-k}(D_J)
\prod_{j \in J} \left( \frac{-a_j}{a_j+1}\right).
\end{align*}
We give a proof for its independence  
on the choice of $\rho$ without  
using the definition of stringy Chern classes. 

In Section \ref{section3}, we apply the results of the previous sections to 
hypersurfaces and complete intersections in normal
projective $\Q$-Gorenstein toric varieties. 
For this purpose, we compute the $k$-th stringy Chern class of
a $d$-dimensional projective $\Q$-Gorenstein 
toric variety $X$ associated with a fan $\Sigma$ 
of rational polyhedral cones in $N_{\R}$
as a linear combination of classes of 
torus-invariant cycles $X_\sigma$ 
corresponding to $k$-dimensional cones, i.e.,
\begin{align*} 
c_{k}^{str} (X) = \sum_{\sigma \in \Sigma(k)} v(\sigma) 
\cdot \left[X_{\sigma}\right],
\end{align*}
where $\Sigma(k)$ is the set of all $k$-dimensional cones of $\Sigma$
and $N_{\R}$ is a real vector space obtained by an 
extension of a $d$-dimensional lattice $N \cong \Z^d$.
The coefficients $v(\sigma)$   
are positive integers defined by the formula 
\begin{align*}
v(\sigma)\defeq k! \cdot \vol_{k} \left(\Theta_{\sigma}\right), 
\end{align*}
where  $\vol_{k} \left(\Theta_{\sigma}\right)$ is the 
$k$-dimensional volume 
of the lattice polytope $\Theta_{\sigma}$ obtained as the convex hull 
of the origin
and the primitive lattice generators of all $1$-dimensional 
faces of $\sigma$ with respect to the sublattice 
$\left< \sigma \right>_{\R} \cap N$. 
We apply this formula for stringy 
Chern classes of toric varieties to compute 
intersection numbers
$[D_1].\; \ldots \; .[D_k].c_{d-k}^{str} (X)$
via mixed volumes of faces of divisor-associated convex lattice polytopes,
where $D_1, \ldots, D_k$ are semiample torus-invariant $\Q$-Cartier divisors.

In Section \ref{section4}, we are interested in a combinatorial interpretation 
of the stringy Libgober-Wood identity \eqref{eq1} 
for $d$-dimensional projective $\Q$-Gorenstein toric varieties. 
Such a toric
variety $X$ is defined by a complete 
fan $\Sigma$ in $ N_{\R}$ such that 
there exists 
a piecewise linear function $\kappa: N_{\R} \rightarrow \R$ 
that is linear on each cone $\sigma$ of $\Sigma$ and has 
value $-1$ on every 
primitive lattice generator of a $1$-dimensional cone of $\Sigma$. A
projective $\Q$-Gorenstein 
toric variety  $X$ is called {\em log-Fano toric variety}
if its anticanonical
divisor $-K_X$ is ample, i.e.,  if $\kappa$
is strictly convex. In this case  
\begin{align*}
\Delta \defeq \{ x \in N_{\R} \; \vert \; \kappa(x) \geq -1 \} 
\end{align*} 
is a convex lattice polytope whose vertices are primitive 
lattice generators of 
$1$-dimensional cones of $\Sigma$. 
Let $M \defeq \Hom(N, \Z)$ be the dual lattice  to $N$
and $\langle \cdot,\cdot \rangle : M \times N \to \Z$ 
the natural pairing,
which extends to a pairing 
$\langle \cdot,\cdot \rangle : M_{\R} \times N_{\R} \to \R$.
The {\em dual polytope} $\Delta^*$ of $\Delta$ 
is a convex polytope with rational vertices
defined as 
\begin{align*}
\Delta^* \defeq \{ y \in M_\R  \; \vert \; \left< y,x \right> \geq -1 \; 
 \forall x \in \Delta\}.
\end{align*}  
The polytope $\Delta$ is called {\em reflexive} if all vertices
of the dual polytope $\Delta^*$ belong to the lattice $M$. 
The latter case happens if and
only if the log-Fano toric variety $X$ 
is a {\em Gorenstein toric Fano variety} 
(i.e., its anticanonical
divisor is an ample Cartier divisor). 
Gorenstein toric Fano varieties $X$ of dimension $d$ 
associated with reflexive polytopes $\Delta$ are used in mirror symmetry 
as ambient spaces
for Calabi-Yau hypersurfaces \cite{Bat94, BB97, BD96}. 
The stringy Euler number 
of a generic ample Calabi-Yau hypersurface $Z$ in $X$ 
is combinatorially computable via
\begin{align*}
e_{str}(Z)= c_{d-1}^{str} \left( Z \right)
= \sum_{k=0}^{d-3}  (-1)^{k}  
\sum_{\theta \preceq \Delta \atop \dim\left(\theta \right)=k+1}
v(\theta) \cdot v\left(\theta^*\right),
\end{align*} 
where the face 
$\theta^{*} \defeq \{y \in \Delta^{*} \vert 
\left< y,x\right> = -1 \; \forall x \in \theta \}$
of the dual reflexive polytope $\Delta^{*}$ 
is called {\em dual face} to a face $\theta$ of $\Delta$
(written $\theta \preceq \Delta$).

If $X$ is a toric log del Pezzo surface associated to a
convex lattice polytope $\Delta$
(cf. LDP-polytope \cite{KKN10}), then the 
stringy Libgober-Wood identity \eqref{eq1} for $X$  
is equivalent to the combinatorial equality 
\begin{align*}
v\left(\Delta \right) +  v \left(\Delta^* \right)
= 12 \sum_{n \in \Delta \cap N} \left( \kappa(n) + 1 \right)^2.
\end{align*}
In particular, 
one always has $v\left(\Delta \right) +  v \left(\Delta^* \right) \geq 12$
and equality holds, i.e.,
\begin{align} \label{d2}
v(\Delta) + v\left( \Delta^{*} \right)=12, 
\end{align}
if and only if 
$\Delta$ is a reflexive polytope.

If $\Delta$ is a $3$-dimensional reflexive polytope, then the stringy Euler
number of a Calabi-Yau hypersurface in the associated Gorenstein toric variety $X$ is 
$24$ and one obtains the identity 
\begin{align} \label{d3}
\sum_{\theta \preceq \Delta \atop \dim\left(\theta\right)=1} 
v\left(\theta\right) \cdot v\left(\theta^{*}\right)=24.
\end {align}

Our aim in Section \ref{section5} is to give  a generalization of 
the above identities for reflexive polytopes of dimension $d \geq 4$. 
For this purpose, we use the 
{\em Ehrhart power series} of an arbitrary $d$-dimensional 
lattice polytope $\Delta \subseteq N_{\R}$ defined as 
\begin{align*}
P_{\Delta}(t) \defeq \sum_{k\geq 0} 
\left\vert k\Delta \cap N \right\vert t^k, 
\end{align*}
where $\left\vert k\Delta \cap N \right\vert$ denotes the 
number of lattice points in $k\Delta$. This series 
can be written as the rational function 
\begin{align*} 
P_{\Delta}(t) =\frac{\psi_d\left(\Delta\right) t^d + 
\ldots+\psi_1\left(\Delta\right)t
+\psi_0\left(\Delta\right)}{(1-t)^{d+1}},
\end{align*}
where 
$\psi_{\alpha} \left(\Delta\right)$ 
are nonnegative integers for all $0 \leq \alpha \leq d$
\cite[Theorem 2.11]{Bat93}.
We show that the stringy Libgober-Wood 
identity \eqref{eq1} 
for a Gorenstein toric Fano 
variety defined by the fan of cones 
over faces of a reflexive polytope 
$\Delta$ is equivalent
to the combinatorial identity
\begin{align} \label{sLWrefl} 
\sum_{\alpha \in [0,d] \cap \Z} \psi_{\alpha}\left(\Delta\right)  
\left(\alpha-\frac{d}{2} \right)^{2}= 
\frac{d}{12} v\left(\Delta\right)+ \frac{1}{6} 
\sum_{\theta \preceq \Delta \atop  \dim(\theta)=d-2} 
v(\theta) \cdot v \left(\theta^{*}\right).
\end{align}

For reflexive polytopes $\Delta$ 
of dimension $2$ and $3$ this identity
is equivalent to 
Equation \eqref{d2} respectively \eqref{d3}, but  
for reflexive polytopes $\Delta$ of dimension 
$d \geq 4$ 
Equation \eqref{sLWrefl} 
is not anymore symmetric with respect to the polar duality between 
$\Delta$ and $\Delta^*$ (i.e., the equality for $\Delta$ is not
equivalent to the one for $\Delta^*$). If $d=4$, then 
\eqref{sLWrefl} is equivalent to 
\begin{align*} 
12 \cdot \left\vert \partial 
\Delta \cap N \right\vert= 2 \cdot v\left( \Delta \right) 
+  
\sum_{\theta \preceq \Delta \atop  \dim(\theta)=2} 
v(\theta) \cdot v \left(\theta^{*}\right),
\end{align*}
where $\partial \Delta$ denotes the boundary of the polytope $\Delta$
and $\left\vert \partial 
\Delta \cap N \right\vert$ the number of lattice points in $\partial \Delta$.
In addition, we consider some generalizations of the above identities 
for Gorenstein polytopes. 
  
\section{Stringy Chern classes of complete intersections} \label{section1}
First, we note that for projective varieties $X$ 
the general definition of the total stringy Chern class 
$c_{\sbullet}^{str}(X) \in A_{\sbullet}(X)_\Q$  
due to de Fernex, Luperico, Nevins, and Uribe \cite{dFLNU07} can be 
simplified as follows:

\begin{propo} \label{new-string}
Let $X$ be a $d$-dimensional normal projective $\Q$-Gorenstein variety 
with at worst log-terminal singularities and 
$\rho : Y \rightarrow X$ 
a log-desingularization of $X$. 
Then the total stringy Chern class of $X$ can be computed via 
total Chern classes $c_{\sbullet} \left(D_J \right)$
of smooth projective subvarieties $D_J \subseteq Y$ by  
\begin{align} \label{strtc}
 c_{\sbullet}^{str}(X) \defeq \rho_* 
\bigg( \sum_{\emptyset \subseteq J \subseteq I} {e_J}_* 
c_{\sbullet} \left(D_J \right) \prod_{j \in J} 
\left( \frac{-a_j}{a_j+1}\right) \! \! \bigg)   \in A_{\sbullet}(X)_\Q.
\end{align}
In particular, 
one obtains
\begin{align} \label{strck}
c_{k}^{str}(X)\defeq \rho_* \bigg(
\sum_{\emptyset \subseteq J \subseteq I} 
{e_J}_* c_{k-\lvert J \rvert}(D_J) 
\prod_{j \in J} \left( \frac{-a_j}{a_j+1}\right) \! \! \bigg)   \in A_{d-k}(X)_\Q
\end{align}
and has
\begin{align*} 
c_{k}^{str}(X)= \rho_* c_k(Y)  \in A_{d-k}(X)_\Q
\end{align*}
if $\rho:  Y \to X$ is a crepant 
log-desingularization $(0 \leq k \leq d)$.
\end{propo} 

\begin{proof} The definition of 
the total stringy 
Chern class given by  de Fernex, Luperico, Nevins, and Uribe in  
\cite{dFLNU07}
uses the group homomorphism 
of  MacPherson $c: F(Y) \rightarrow A_{\sbullet}(Y)$
from the group $F(Y)$ 
of constructible functions on $Y$ to the Chow group $A_{\sbullet}(Y)$ of $Y$.
If  $\mathds{1}_Y$ is the characteristic function of the 
smooth variety $Y$, then
$c( \mathds{1}_Y) = c_{\sbullet}(Y)$. Using 
the stratification of $Y$  by locally closed subsets 
$D_J^{\circ}:= D_J \setminus \left(\cup_{i \in I \setminus J} D_i \right)$  
$(\emptyset \subseteq J \subseteq I)$ their definition looks as follows:
\begin{align} \label{fernex}
c_{\sbullet}^{str}(X)\defeq  \rho_*   \bigg( \sum_{\emptyset \subseteq J \subseteq I} 
c \left(\mathds{1}_{D_J^{\circ}} \right) \prod_{j \in J} 
\left(\frac{1}{a_j+1}\right) \! \! \bigg).
\end{align}
Using the stratification $\mathds{1}_{D_J}= 
\sum_{J' \supseteq J} \mathds{1}_{D_{J'}^{\circ}}$ 
and 
$c \left( \mathds{1}_{D_J} \right)= 
\sum_{J' \supseteq J} c \left(\mathds{1}_{D_{J'}^{\circ}} \right)$ for all 
$\emptyset \subseteq J \subseteq I$, we conclude
\begin{align*}
\sum_{\emptyset \subseteq J \subseteq I} 
c \left(\mathds{1}_{D_J^{\circ}} \right)\prod_{j \in J} 
\left( \frac{1}{a_j+1}\right) & = \sum_{\emptyset \subseteq J \subseteq I} 
\bigg( \sum_{J' \supseteq J} c \left(\mathds{1}_{D_{J'}^{\circ}} \right)\! \! \bigg)  
\prod_{j \in J} \left( \frac{1}{a_j+1} -1\right) \\ & = 
\sum_{\emptyset \subseteq J \subseteq I} 
c\left(\mathds{1}_{D_J}\right) \prod_{j \in J} 
\left( \frac{-a_j}{a_j+1}\right).
\end{align*}
It remains 
to apply $\rho_*$ to the above equality and the property   
 $c \left(\mathds{1}_{D_J}\right)= {e_J}_*  
 c_{\sbullet} \left( D_J \right) $, which follows from the commutative 
diagram
\begin{center}
\makebox[0pt]{
\begin{xy}
\xymatrix{
F \left(D_J \right) \ar[r]^{}   \ar[d]_{c} & F \left(Y \right) \ar[d]^c \\
A_{\sbullet}\left(D_J \right) \ar[r]_ {{e_J}_*} &   A_{\sbullet}\left(Y \right)  ,
}
\end{xy}
}
\end{center}
where ${e_J}_* : A_{\sbullet}(D_J) 
\to A_{\sbullet}(Y)$ is the push-forward homomorphism corresponding to the 
closed embedding $ {e_J}: D_J \hookrightarrow Y$. 

\end{proof}

\begin{rem} {\rm
The important ingredient in
the definition of the total stringy  Chern class 
$c_{\sbullet}^{str}(X)$  (resp. $k$-th stringy Chern class 
$c_{k}^{str}(X)$)  is its independence  
on the choice of the log-desingularization $\rho$. This property 
was proved in } \cite[Proposition 3.2]{dFLNU07}. 
\end{rem}

Let $V$ be a smooth $d$-dimensional normal projective
variety and $Z$ a smooth divisor on $V$ with 
the closed embedding $i :  Z \hookrightarrow V$.
Using the exact sequence of vector bundles
$ 0 \to \calt_Z \to i^* \calt_V \to {\mathcal O}_Z(Z) \to 0 ,$
one obtains a formula that 
computes the total Chern class of $Z$ in terms of the total
Chern class of the ambient variety $V$:
\begin{align} \label{ismooth}
i_* c_{\sbullet} (Z) = c_{\sbullet}(V) .[Z]( 1+ [Z])^{-1} 
=  c_{\sbullet}(V). \bigg( \sum_{k=1}^{\infty} (-1)^{k-1} [Z]^k \bigg) ,
\end{align}
where $\calt_{\cdot}$ denotes the tangent bundle,
$[Z]$ the class of $Z$ in $A_{d-1}(V)$,
and $c_{\sbullet} (V) \defeq \sum_k c_k(V)$ the total Chern class of $V$
(cf. \cite[Example 3.2.12]{Ful98}). 
In particular, the Euler number of $Z$ 
can be computed as 
\begin{align} \label{ez}
e(Z) = c_{d-1}(Z) =  \sum_{k=1}^{d}(-1)^{k-1} [Z]^k. c_{d-k}(V). 
\end{align}

We show  that the same formulas 
hold for the total stringy Chern class respectively the top stringy Chern class
of generic hypersurfaces in singular varieties.

\begin{theo} \label{thesing2}
Let $X$ be a normal projective $\Q$-Gorenstein variety 
with at worst log-terminal 
singularities and 
$Z$ a generic semiample Cartier divisor on $X$.
Then
the total stringy Chern class of $Z$ is 
\[ i_*  c_{\sbullet}^{str} (Z)= c_{\sbullet}^{str}(X) .[Z] \left(1+ [Z]\right)^{-1}, \]
where $i :  Z \hookrightarrow X$ is the closed embedding.
\end{theo}

\begin{proof}
Let  $\rho : Y \rightarrow X$ be a log-desingularization of $X$ and
$K_Y=\rho^* K_X + \sum_{i=1}^s a_iD_i$. 
By Theorem of Bertini, $Z' \defeq \rho^{-1}(Z)$ is a smooth
divisor on $Y$. By the  
 adjunction formula  
$K_{Z'} = \big( K_Y + Z' \big)  \vert_{Z'}$,  
we obtain 
$ K_{Z'}=\rho^*K_Z + \sum_{i=1}^s a_i D_i', $
where $D_i' \defeq D_i\cap Z'$.
Define $D_J '\defeq \cap_{j \in J} D_j'$ and
 note that $D_J' = D_J \cap Z' \subseteq D_J$
is a smooth divisor on $D_J$. Let ${e'_J}_* : A_{\sbullet}\left(D'_J\right) 
\to A_{\sbullet}\left(Z'\right)$ be the push-forward homomorphism corresponding to
 the closed embedding
$e_J': D_J' \hookrightarrow Z'$. 
Consider the commutative diagram
\begin{center}
\makebox[0pt]{
\begin{xy}
\xymatrix{
D_J  \ar@{^{(}->}[r]^{e_J}^{e_J} & Y \ar[r]^{\rho} & X \\
D_J' \ar@{^{(}->}[r]_{e_J'} \ar@{^{(}->}[u]_{i_J} & Z' \ar[r]_{\rho_Z}  & Z  , \ar@{^{(}->}[u]_i
}
\end{xy}
}
\end{center}
where $\rho_Z: Z' \rightarrow Z$ and $i_J: D'_J \rightarrow D_J$ 
are restrictions of $\rho$ respectively $i$.
We apply  Equation \eqref{ismooth} to the smooth divisor  
$D_J' \subseteq D_J$ and obtain
\begin{align*}
i_* {\rho_Z}_*   {e_J'}_*  
c_{\sbullet} \big( D_J' \big)
= \rho_*  {e_J}_*   {i_J}_* 
c_{\sbullet} \big( D_J' \big)  
= \rho_*  {e_J}_*  \big( c_{\sbullet} \left( D_J \right).
\big[D_J'\big] \big( 1+ \big[D_J'\big] \big)^{-1}  \big).
\end{align*}
Using the projection formula twice provides
\begin{align*}
{e_J}_*  \big( c_{\sbullet} \left( D_J \right).
\big[D_J'\big] \big(1+ \big[D_J'\big] \big)^{-1} \big)
={e_J}_*  c_{\sbullet} \left( D_J \right). \big[Z'\big]
\big(1+ \big[Z'\big] \big)^{-1}
\end{align*}
and
\begin{align*}
\rho_* \big( {e_J}_* c_{\sbullet} \left( D_J \right). \big[Z' \big]
\big(1+ \big[Z'\big] \big)^{-1}\big)
= \rho_*   {e_J}_*  c_{\sbullet} \left( D_J \right).[Z] \left(1+ [Z]\right)^{-1} 
\end{align*}
because ${e_J}^* Z' =D_J'$, $\rho^*  Z  = Z'$ 
and the pull-backs ${e_J}^*, \rho^*$ are homomorphisms.
Therefore, we get
\begin{align*}
 i_* \rho_*  {e_J'}_*  
 c_{\sbullet}  \big(D_J' \big) 
= \rho_*   {e_J}_*c_{\sbullet} \left( D_J \right) .[Z] \left(1+ [Z]\right)^{-1} .
\end{align*}
By applying Proposition \ref{new-string} to $Z$ and $X$,  we  
conclude
\begin{align*}
i_*  c_{\sbullet}^{str}(Z)
&=  i_*  \rho_*   \bigg( \sum_{\emptyset \subseteq J \subseteq I} 
{e_J'}_* \bigg( c_{\sbullet} \left( D_J' \right)
\prod_{j \in J} \left( \frac{-a_j}{a_j+1}\right) \! \! \bigg) \! \! \bigg)  \\
&= \rho_* \bigg( \sum_{\emptyset \subseteq J \subseteq I} 
{e_J}_*  c_{\sbullet} \left( D_J\right) 
\prod_{j \in J} \left( \frac{-a_j}{a_j+1}\right) \! \! \bigg)   .[Z] \left(1+ [Z]\right)^{-1} \\
&= c_{\sbullet}^{str} (X). \left[Z\right] \left( 1+ [Z]\right)^{-1}.
\end{align*}
\end{proof}

\begin{coro} \label{esing2}
Let $X$ be a normal projective $\Q$-Gorenstein variety of dimension $d$ 
with at worst log-terminal 
singularities and 
$Z$ a generic semiample Cartier divisor on $X$. Then
the stringy Euler number of $Z$ is
\begin{align*}
e_{str}(Z) = c_{d-1}^{str}(Z) =  
\sum_{k=1}^{d}(-1)^{k-1} [Z]^k. c_{d-k}^{str}(X).
\end{align*}
If $\left[Z \right] = c_1 \left(X \right)$, 
$Z$ has trivial anticanonical class and 
the formula simplifies to
\begin{align*}
e_{str}(Z) = c_{d-1}^{str}(Z) = \sum_{k=1}^{d-2}(-1)^{k-1} [Z]^k. c_{d-k}^{str}(X).
\end{align*}
\end{coro}

The  formulas in Theorem \ref{thesing2} and
Corollary \ref{esing2} for 
total stringy Chern classes 
respectively top stringy Chern classes
of generic hypersurfaces can be generalized to
 generic complete intersections in singular varieties:  

\begin{theo} 
Let $X$ be a normal projective $\Q$-Gorenstein variety with at worst log-terminal 
singularities, $Z_1,\ldots,Z_r$ generic semiample Cartier divisors on $X$,
and $i :  Z_1 \cap \ldots \cap Z_r  \hookrightarrow X$ the closed embedding. Then
the total stringy Chern class of the complete intersection $Z_1 \cap \ldots \cap Z_r$ is
\begin{align*}
i_* c_{\sbullet}^{str} \left(Z_1 \cap \ldots \cap Z_r \right) = 
c_{\sbullet}^{str}(X).
\prod_{j=1}^r  \left[Z_j\right]\left(1+ \left[Z_j\right]\right)^{-1}. 
\end{align*}
\end{theo}

\begin{proof} We apply induction on $r$ and use Theorem \ref{thesing2} 
$r$-times, since for any $2 \leq r' \leq r$ the complete intersection 
$Z_1 \cap \ldots \cap Z_{r'}$ is a generic hypersurface in 
$Z_1 \cap \ldots \cap Z_{r'-1}$.
\end{proof}

\begin{coro} \label{corcompl}
Let $X$ be a normal projective $\Q$-Gorenstein variety of dimension $d$ 
with at worst log-terminal 
singularities and 
$Z_1,\ldots,Z_r$ generic semiample Cartier divisors on $X$.
Then the stringy Euler number $c_{d-r}^{str}\left(Z_1\cap \ldots \cap Z_r\right)$
of the complete intersection $Z_1\cap \ldots \cap Z_r$ is
\begin{align*}
\sum_{k=0}^{d-r} (-1)^k \left[Z_1\right]. \; \ldots \;. \left[Z_r\right].
\bigg( \sum_{j_0,\ldots,j_k=1 \atop j_0 \leq \ldots \leq j_k}^r 
\left[Z_{j_0}\right]. \; \ldots \; .  \left[Z_{j_k}\right] \bigg) 
. c_{d-r-k}^{str}\left(X\right).
\end{align*}
\end{coro}

\begin{coro} \label{corcompl2}
Let $X$ be a normal projective $\Q$-Gorenstein variety of dimension $d$ 
with at worst log-terminal 
singularities and 
$Z_1,\ldots,Z_r$ generic semiample Cartier divisors on $X$ 
such that 
$\left[Z\right] \defeq \left[Z_1\right]=\ldots= \left[Z_r\right]$.
Then the stringy Euler number $c_{d-r}^{str}\left(Z_1\cap \ldots \cap Z_r\right)$
of the complete intersection $Z_1\cap \ldots \cap Z_r$ is
\begin{align*}
c_{d-r}^{str}\left(Z_1\cap \ldots \cap Z_r\right) 
&= \sum_{k=0}^{d-r} (-1)^k \binom{k+r-1}{r-1} 
\left[Z\right]^{r+k}. c_{d-r-k}^{str}\left(X\right).
\end{align*}
\end{coro}

\section{Intersection numbers with stringy Chern classes} \label{section2}
Let $X$ be a $d$-dimensional normal projective $\Q$-Gorenstein variety with
at worst log-terminal singularities,
$\rho : Y \rightarrow X$ 
a log-desingularization of $X$, and 
$Z_1, \ldots, Z_k$ arbitrary $\Q$-Cartier divisors 
on $X$. Intersecting the classes $[Z_1], \ldots, [Z_k]\in {\rm Pic }(X)_{\Q}$ with 
the stringy Chern class $c_{d-k}^{str} (X) \in A_k(X)_\Q$,
one obtains 
the rational intersection number  $[Z_1].\; \ldots\; .[Z_k].c_{d-k}^{str} (X)$,
which can be considered as a generalization of Equation \eqref{eq12}.
Using \eqref{strck} for the definition of the stringy Chern class $c_{d-k}^{str} (X) \in A_k(X)_\Q$ 
and the projection formula for the proper morphism 
$\rho$, one receives
\begin{align*}
[Z_1]. \;\ldots \; .[Z_k].c_{d-k}^{str} (X)  
&= [Z_1].\;\ldots \;.[Z_k]. 
\rho_* \bigg(\sum_{\emptyset \subseteq J \subseteq I}  \! \!
{e_J}_* c_{d-\lvert J \rvert-k}(D_J)
\prod_{j \in J} \left( \frac{-a_j}{a_j+1}\right) \! \!
\bigg)\\
&=\rho_* \bigg( \rho^*   \left( [Z_1].\;\ldots \;.[Z_k] \right).  \! \!  \! \!
\sum_{\emptyset \subseteq J \subseteq I}  \! \! 
{e_J}_* c_{d-\lvert J \rvert-k}(D_J)
\prod_{j \in J} \left( \frac{-a_j}{a_j+1}\right)  \! \! \bigg)\\
&=\sum_{\emptyset \subseteq J \subseteq I}
\rho^*[Z_1]. \; \ldots\; . \rho^*[Z_k].{e_J}_* c_{d-\lvert J \rvert-k}(D_J)
\prod_{j \in J} \left( \frac{-a_j}{a_j+1}\right) ,
\end{align*} 
where the homomorphism  $\rho^* : {\rm Pic }(X)_{\Q} \to {\rm Pic}(Y)_{\Q}$ is determined 
by the pullback of line bundles  
and the intersection product 
\begin{align*}
\rho^*[Z_1].\; \ldots\; . \rho^*[Z_k]. {e_J}_* c_{d-\lvert J \rvert-k}(D_J)
\end{align*}
is the value of 
the multilinear map
${\rm Pic }(Y)_{\Q}^k \times A_k(Y) \to \Q$
defined by the intersection of $k$ classes $\rho^*[Z_1], 
\ldots, \rho^*[Z_k]$ of $\Q$-Cartier divisors on $Y$ with 
the $k$-dimensional cycle  ${e_J}_*c_{d-\lvert J \rvert-k}(D_J) \in 
A_{k}(Y)$ composed with the natural map $A_0(Y)_\Q \to \Q$. 
In particular, we obtain 

\begin{theo} \label{intersect}
The rational intersection number 
\[ \sum_{\emptyset \subseteq J \subseteq I}
\rho^*[Z_1].\; \ldots\; . \rho^*[Z_k]. {e_J}_*c_{d-\lvert J \rvert-k}(D_J)
\prod_{j \in J} \left( \frac{-a_j}{a_j+1}\right) \]
does not depend on the choice of the log-desingularization 
$\rho :  Y \to X$. 
\end{theo}

\begin{coro} \label{c_1}
Let $\rho^*c_1(X) = [-\rho^* K_X]$ be the pullback of the anticanonical
class of $X$. Then the intersection number 
\[ \sum_{\emptyset \subseteq J \subseteq I} 
\rho^{*}c_1(X)^k . {e_J}_* c_{d-\lvert J \rvert-k}\left(D_J \right) 
\prod_{j \in J} \left( \frac{-a_j}{a_j+1}\right) \]
is independent of the log-desingularization $\rho$.
\end{coro}

In the case  $k=1$   (cf. Equation \eqref{eq12}), 
this theorem has been proved  in 
\cite[Corollary 3.9]{Bat00}.

We observe that the intersection number  
$[Z_1].\; \ldots\; .[Z_k].c_{d-k}^{str} (X)$  can be computed by a formula
that does not 
involve stringy Chern classes of the singular variety $X$, but 
only usual Chern classes of smooth subvarieties  $D_J \subseteq Y$. 
We give below a proof of Theorem \ref{intersect}, which does not 
use the general definiton of stringy Chern classes, but 
only the definition 
of the stringy Euler number \eqref{estr} and its independence of 
log-desingularization.  

\begin{proof}[Proof of Theorem \ref{intersect}]
Let us denote the number 
\[ \sum_{\emptyset \subseteq J \subseteq I}
\rho^*[Z_1]. \; \ldots \;. \rho^*[Z_k]. {e_J}_* c_{d-\lvert J \rvert-k}(D_J)
\prod_{j \in J} \left( \frac{-a_j}{a_j+1}\right) \]
by $\imath_\rho( Z_1, \ldots, Z_k )$. It is clear 
that the map
\[ \imath_\rho : {\rm Pic}(X)_{\Q}^k \to \Q ,
([Z_1], \ldots, [Z_k]) \mapsto \imath_\rho ( Z_1, \ldots, Z_k ) \]
is symmetric and  multilinear. Since the group ${\rm Pic}(X)$ is generated by 
classes of very ample Cartier divisors, it is enough to show the statement of 
Theorem \ref{intersect} in the case when $Z_1, \ldots, Z_k$ are very ample 
Cartier divisors. For any sequence of positive 
integers $n_1, \ldots, n_k$, the linear 
combination $n_1 [Z_1] + \cdots + n_k [Z_k]$ represents a class of 
a very ample Cartier divisor $Z$ on $X$. It follows from the symmetry 
and multilinearity 
of $\imath_\rho$ that $\imath_\rho(Z, Z, \ldots, Z)$ is a homogeneous polynomial
of degree $k$  in $n_1, \ldots, n_k$ whose coefficients are 
the rational numbers $\imath_\rho(Z_{i_1},Z_{i_2}, \ldots, Z_{i_k})$, where 
$1 \leq i_1, \ldots, i_k \leq k$. 
Therefore, it is enough to 
prove the statement of Theorem \ref{intersect} only for the rational numbers
 $\imath_\rho(Z, Z, \ldots, Z)$, where $Z$ is a generic very ample Cartier divisor on $X$.
By Theorem of Bertini, we can assume that $Z'\defeq \rho^{-1}(Z)$ is a smooth 
divisor on $Y$ and the restriction of $\rho$ to $Z'$ defines a log-desingularization
of $Z$ with the exceptional divisors $D_i'\defeq D_i \cap Z'$ such that 
$K_{Z'} = \rho^* K_Z + a_1D_1' + \cdots + a_s D_s'$.  
One can compute the stringy Euler number of $Z$ by 
\begin{align*}
 e_{str}(Z) = \sum_{\emptyset \subseteq J \subseteq I} e(D_J')
\prod_{j \in J} \left( \frac{-a_j}{a_j+1}\right), 
\end{align*}
where $e(D_J')$ denotes the usual Euler number of the smooth variety 
$D_J'= D_J \cap Z'$ (cf. Equation \eqref{estr} and \cite[Definition 3.3]{Bat98}). 
Now we apply Equation \eqref{ez} to each smooth divisor 
$D_J' \subseteq D_J$ and 
obtain
\begin{align*}
e(D_J') =\sum_{k =1}^{d- |J|} (-1)^{k-1} \left[D_J'\right]^k.  c_{d- |J|-k}(D_J) 
= \sum_{k =1}^{d- |J|} (-1)^{k-1} \rho^* \left[Z\right]^k. {e_J}_* c_{d- |J|-k}(D_J)
\end{align*}
because the projection formula for the proper morphism 
$e_J: D_J \hookrightarrow Y$ implies
\begin{align*}
\rho^* \left[Z\right]^k. {e_J}_* c_{d- |J|-k}(D_J)
&={e_J}_* ( {e_J}^* \rho^* \left[Z\right]^k.  c_{d- |J|-k}(D_J)) 
= {e_J}^* \rho^* \left[Z\right]^k. c_{d- |J|-k}(D_J) \\
&= {e_J}^* \left[ Z' \right]^k.  c_{d- |J|-k}(D_J) 
= \left[ D_J' \right]^k.  c_{d- |J|-k}(D_J) .
\end{align*}
Therefore, the stringy Euler number $e_{str}(Z)$ of $Z$ has the form
\begin{align*}
 e_{str}(Z) 
 &= \sum_{\emptyset \subseteq J \subseteq I} \bigg( \sum_{k =1}^{d- |J|} 
(-1)^{k-1} \rho^* \left[Z\right]^k. {e_J}_* c_{d- |J|-k}(D_J)
\bigg) \prod_{j \in J} \left( \frac{-a_j}{a_j+1}\right) \\
&= \sum_{k \geq 1} (-1)^{k-1}\rho^*[Z]^k \bigg( \sum_{\emptyset \subseteq J \subseteq I}
{e_J}_* c_{d-\lvert J \rvert-k}(D_J)
\prod_{j \in J} \left( \frac{-a_j}{a_j+1}\right) \! \! \bigg) .
\end{align*}
For any positive integer $n$, the class $n [Z]$  can be again represented by 
a generic very ample Cartier divisor $Z^{(n)}$ such that $\rho^{-1}(Z^{(n)})$ is smooth. So 
we can repeat the same arguments for $Z^{(n)}$ to obtain that the stringy Euler number
\begin{align*}
e_{str}\left(Z^{(n)}\right) 
= \sum_{k \geq 1} (-1)^{k-1}n^k\rho^*[Z]^k 
\bigg( \sum_{\emptyset \subseteq J \subseteq I}
 {e_J}_*c_{d-\lvert J \rvert-k}(D_J)
\prod_{j \in J} \left( \frac{-a_j}{a_j+1}\right) \! \! \bigg)  
\end{align*}
is a polynomial $P$ in $n$ whose $k$-th coefficient 
\begin{align*}
\rho^*[Z]^k 
\bigg( \sum_{\emptyset \subseteq J \subseteq I}
 {e_J}_*c_{d-\lvert J \rvert-k}(D_J)
\prod_{j \in J} \left( \frac{-a_j}{a_j+1}\right) \! \! \bigg)  
\end{align*} 
is equal to $\imath_\rho(Z, \ldots,Z)$. 
Since the stringy Euler number of $Z^{(n)}$ does not depend on the choice of 
the log-desingularization $\rho$ \cite[Theorem 3.4]{Bat98},
the same is true for the polynomial $P$ and hence for its $k$-th coefficient 
$\imath_\rho(Z, \ldots,Z)$. 
\end{proof}

\section{Stringy Chern classes on toric varieties} \label{section3}
It is well-known that singularities 
of an arbitrary  $\Q$-Gorenstein toric variety $X$  are log-terminal.
Moreover, the stringy Euler number $e_{str}(X)$ of the 
 toric variety $X$
can be computed combinatorially using cones of maximal 
dimension in the associated fan $\Sigma$ \cite[Proposition 4.10]{Bat98} 
In this section, we give a combinatorial formula of  all 
stringy Chern classes 
of arbitrary $\Q$-Gorenstein toric varieties 
using the  intrinsic information provided by the associated fans. 
We apply this formula to compute intersection numbers 
$\left[D_1\right].  \; \ldots \; . \left[D_k\right].c_{d-k}^{str}(X)$ 
via mixed volumes of certain polytopes.

We start with a well-know fact about Chern classes of 
smooth toric varieties:

\begin{theo} \label{torvol-smooth}
Let $V$ be a smooth  
toric variety
associated with a fan $\Sigma$ in $N_{\R}$. 
Then the total Chern class of $V$ is 
\begin{align*}
c_{{\sbullet}} (V) = \sum_{\sigma \in \Sigma}  
\left[V_{\sigma}\right],
\end{align*}
where $\left[V_{\sigma}\right]$ is the class 
of the closed torus orbit $V_\sigma$ corresponding to a cone 
$\sigma \in \Sigma$.
\end{theo} 

\begin{theo} \label{torvol}
Let $X$ be a $\Q$-Gorenstein 
toric variety
associated with a fan $\Sigma$. 
Then the total stringy Chern class of $X$ is 
\begin{align*}
c_{{\sbullet}}^{str} (X) = \sum_{\sigma \in \Sigma} v(\sigma) 
\cdot \left[X_{\sigma}\right],
\end{align*}
where $v(\sigma) = k! \cdot \vol_{k} \left(\Theta_{\sigma}\right)$
and   $\vol_{k} \left(\Theta_{\sigma}\right)$ is the 
$k$-dimensional volume 
of the lattice polytope $\Theta_{\sigma}$ obtained as the convex hull 
of the origin
and the primitive lattice generators of all $1$-dimensional 
faces of a cone $\sigma \in \Sigma(1)$. 
\end{theo} 

\begin{proof}
Consider $\rho : Y \rightarrow X$ to be a 
log-desingularization of $X$ obtained by a refinement
$\Sigma'$ of the fan $\Sigma$. 
There is a natural 
bijection between the set of exceptional 
divisors $\{D_1, \ldots, D_s\}$ in $Y$
and the set of $1$-dimensional cones 
$\Sigma'(1) \setminus \Sigma(1)$. We denote by 
$\{D_{s+1},\ldots,D_r\}$ the set of all remaining torus-invariant 
divisors in $Y$ whose elements one-to-one correspond to $1$-dimensional 
cones in  $\Sigma(1)$ and set $I \defeq \{1,\ldots,s \}$,  
$I' \defeq I \cup \{s+1,\ldots,r \}$, and 
$a_j \defeq 0$ for all $j \in I' \setminus I$.

Let $\{u_1, \ldots, u_r\} = \{ u_i\, \vert \, i \in I'\}$ 
be the set of all primitive lattice generators of $1$-dimensional
cones of $\Sigma'(1)$ corresponding to all 
torus-invariant divisors $D_1, \ldots, D_r$.
For any subset $J' \subseteq I'$, 
the subset  $D_{J'}=  \cap_{j \in J'} D_j$ is either empty or 
a smooth toric subvariety $Y_{\sigma'}$ 
of $Y$. The latter holds if and only if the set 
$\{ u_j  \vert  j \in J'  \}$ generates a cone
$\sigma' \in \Sigma'$ of dimension $|J'|$. 
Then the locally closed subset
$D^{\circ}_{J'} \defeq D_{J'} \setminus \left( \cup_{i \in I' \setminus J'} 
D_i\right)$ is the dense open torus orbit  $T_{\sigma'}$ in  $Y_{\sigma'}$.
By Theorem \ref{torvol-smooth}, we have 
$c_{\sbullet}(T_{\sigma'})= [T_{\sigma'}]$. This implies 
$c \left(\mathds{1}_{D_{J'}^{\circ}} \right) =  \left[Y_{\sigma'}\right]$.
Using Equation \eqref{fernex},  
we get
\begin{align*} 
c_{\sbullet}^{str}(X)
&= \rho_*   \bigg( \sum_{\emptyset \subseteq J' \subseteq I'} 
c \left(\mathds{1}_{D_{J'}^{\circ}} \right) \prod_{j \in J'} 
\left(\frac{1}{a_j+1}\right) \! \! \bigg)
= \sum_{\sigma' \in \Sigma'} \rho_* \left[Y_{\sigma'} \right] \prod_{u_j \in \sigma'} 
\left( \frac{1}{a_j+1}\right). 
\end{align*}
 
Let us compute $\rho_*   [Y_{\sigma'}]$.   
If $\sigma \in \Sigma$ is the minimal cone of $\Sigma$ 
containing $\sigma' \in \Sigma'$, then $\rho(Y_{\sigma'}) = X_{\sigma}$. 
In order to compute the corresponding cycle map 
$\rho_*: A_{\sbullet}(Y) \to A_{\sbullet}(X)$, we need to compare the dimensions of 
$Y_{\sigma'}$ and $X_{\sigma}$. If $\dim \left( Y_{\sigma'} \right) > \dim \left( X_{\sigma} \right)$, then 
$\rho_*   [Y_{\sigma'}]  =0$. Otherwise, we have
$\rho_*   [Y_{\sigma'}]  = [X_\sigma]$.

Therefore,  we get
\begin{align*} 
c_{\sbullet}^{str}(X) 
&=  \sum_{\sigma \in \Sigma} \bigg(  \sum_{\sigma' \in \Sigma', \sigma' \preceq \sigma \atop  \dim \left( \sigma' \right) = \dim \left( \sigma \right)}
  \prod_{u_j \in \sigma'} 
\left( \frac{1}{a_j+1}\right) \! \! \bigg)  \cdot \left[X_{\sigma} \right].
\end{align*}
Furthermore, $ \prod_{u_j \in \sigma'} \left( \frac{1}{a_j+1}  \right) 
= v\left(\sigma'\right)  $ for every cone $\sigma' \in \Sigma'$ and this implies
\begin{align*} 
c_{\sbullet}^{str}(X) 
= \sum_{\sigma \in \Sigma} \bigg( \sum_{\sigma' \in \Sigma', \sigma' \preceq \sigma \atop  \dim \left( \sigma' \right) = \dim \left( \sigma \right)}
v\left(\sigma'\right) \! \!  \bigg)  \cdot \left[X_{\sigma} \right] 
= \sum_{\sigma \in \Sigma}  v\left(\sigma\right)   \cdot \left[X_{\sigma} \right] .
\end{align*}
\end{proof}

\begin{coro} \label{cortotalchern}
Let $X$ be a $d$-dimensional 
 $\Q$-Gorenstein toric variety  
associated with a fan $\Sigma$. 
Then the $k$-th stringy Chern class of $X$ $(0 \leq k \leq d)$ is 
\begin{align*}
c_{k}^{str} (X) = \sum_{\sigma \in \Sigma(k)} v(\sigma) 
\cdot \left[X_{\sigma}\right].
\end{align*}
\end{coro} 

The above formula allows to compute combinatorially the intersection number of 
the stringy Chern class $c_{d-1}^{str}(X)$ of a  $d$-dimensional 
projective $\Q$-Gorenstein toric variety $X$ corresponding to a
fan $\Sigma$ with 
an arbitrary  torus-invariant $\Q$-Cartier divisor
$D= \sum_{\rho \in \Sigma(1)}a_{\rho} D_{\rho}$ on $X$. 

Therefore, we define  for any $(d-1)$-dimensional 
cone $\sigma \in \Sigma(d-1)$  
the rational number
$l_D(\sigma)$:
Consider two  $d$-dimensional 
cones $\sigma'$, $\sigma'' \in \Sigma(d)$  such that 
$\sigma = \sigma' \cap \sigma''$. Denote by 
$m_{\sigma'}$ and $m_{\sigma''}$  elements in $M_\Q$ that are defined by 
the conditions $\left< m_{\sigma'},u_{\rho} \right> = -a_{\rho} \; \forall 
\rho \subseteq  \sigma'$ respectively
 $\left< m_{\sigma''},u_{\rho} \right> = -a_{\rho} \; \forall 
\rho \subseteq  \sigma''$, 
where $\rho \in \Sigma(1)$ and $u_{\rho} \in N$ denotes its 
primitive lattice generator.
Now choose the primitive lattice generator $u$ of the 
$1$-dimensional sublattice $M(\sigma) \defeq \{ m \in M \vert 
\left< m, u' \right> = 0 \; \forall u' \in \sigma \}$ 
such that $u\vert_{\sigma'} \leq 0$ and $u\vert_{\sigma''} \geq 0$.
Since  $m_{\sigma'}-m_{\sigma''}$ vanishes on $\sigma$, there exists 
a unique number
$l_D(\sigma) \in \Q$
such that $m_{\sigma'}-m_{\sigma''} = l_D(\sigma) \cdot u$.

\begin{prop}  \label{C1Cn-1}
Let $X$ be a $d$-dimensional 
projective $\Q$-Gorenstein toric variety 
associated with a fan $\Sigma$ and 
$D$ a torus-invariant $\Q$-Cartier divisor on $X$. Then
\begin{align*}
\left[D\right].c_{d-1}^{str}(X) 
= \sum_{\sigma \in \Sigma(d-1)}  v\left(\sigma \right) \cdot l_D(\sigma) ,
\end{align*}
where the rational number $l_D(\sigma) \in \Q$ 
is defined as above.
\end{prop}

\begin{proof}
Using Corollary \ref{cortotalchern}, we obtain 
\begin{align*}
\left[D\right].c_{d-1}^{str}(X) 
= \sum_{\sigma \in \Sigma(d-1)} v(\sigma)\cdot 
\left[D\right].\left[X_{\sigma}\right]. 
\end{align*}
It remains to apply the equality 
$\left[D\right].\left[X_{\sigma}\right] = l_D(\sigma)$  
(cf. \cite[Proposition 6.3.8]{CLS11}) for every 
cone $\sigma \in \Sigma(d-1)$.
\end{proof}

Now we compute intersection numbers $[D]^k.c_{d-k}^{str}(X)$, where 
$D = \sum_{\rho \in \Sigma(1)}a_{\rho} D_{\rho}$ is a semiample 
torus-invariant $\Q$-Cartier divisor on the toric variety $X$. 
Consider the corresponding convex rational polytope $\Delta_D$ of dimension
$\leq d$
defined as 
\begin{align*}
\Delta_D \defeq 
\{ y \in M_{ \R} \vert \left<y,u_{\rho} \right> \geq -a_{\rho} \; 
\forall \rho \in \Sigma(1) \} \subseteq M_{\R},
\end{align*}
where $u_{\rho} \in N$ is 
the primitiv lattice generator of a $1$-dimensional cone 
$\rho \in \Sigma(1)$.
Let $\sigma \in \Sigma(d-k)$ be a $(d-k)$-dimensional 
cone of the fan $\Sigma$.  
Denote by  $\Delta_D^{\sigma}$ a
face of $\Delta_D$ of dimension $\leq k$ defined as
\begin{align*}
\Delta_D^{\sigma} \defeq \{ y \in \Delta_D \vert \left<y,u_{\rho} \right> = -a_{\rho} 
\; \forall \rho \in \Sigma(1) \text{ with } \rho \in \sigma \}\subseteq M_{\R}.
\end{align*} 
The  {\em volume of the rational polytope} $\Delta_D^{\sigma}$ is defined as 
\begin{align*}
v\left(\Delta_D^{\sigma} \right) \defeq k! \cdot \vol_k\left(\Delta_D^{\sigma} 
\right) \in \Q,
\end{align*} 
where $\vol_k\left(\Delta_D^{\sigma} \right)$ is the volume of $\Delta_D^{\sigma}$ 
with respect to the $k$-dimensional sublattice 
$M(\sigma) = \{m \in M \vert \left<m,u'  \right> = 0 \; \forall u' \in  
\sigma \}$ of $M$. 
In particular, one has $v\left(\Delta_D^{\sigma} \right) = 0$  
if $\dim \left(\Delta_D^{\sigma} \right) < k$. 

\begin{theo} \label{ZkCn-k}
Let $X$ be a $d$-dimensional 
projective $\Q$-Gorenstein toric variety   
associated with a fan $\Sigma$ and 
$D$ a semiample 
torus-invariant $\Q$-Cartier divisor on $X$. Then
\begin{align*}
\left[D\right]^k.c_{d-k}^{str}(X) = \sum_{\sigma \in \Sigma(d-k)}v(\sigma) 
\cdot v\left(\Delta_D^{\sigma} \right),
\end{align*}
where the face $\Delta_D^{\sigma}$ of $\Delta_D$ and $v\left(\Delta_D^{\sigma} \right)$ 
are defined as above $(0 \leq k \leq d)$.
\end{theo}

\begin{proof}
By Corollary \ref{cortotalchern}, we have 
\begin{align*}
\left[D\right]^k.c_{d-k}^{str}(X) 
= \sum_{\sigma \in \Sigma(d-k)} v(\sigma)\cdot \left[D\right]^k.\left[X_{\sigma}\right].
\end{align*} 
Let $D^{\sigma}$ be the restriction of the semiample torus-invariant $\Q$-Cartier 
divisor 
$D$ to the  $k$-dimensional 
toric subvariety 
$X_{\sigma}$ of $X$. Then  $\left[D\right]^k.\left[X_{\sigma}\right]$ is 
the intersection number $\left[D^{\sigma}\right]^k$ 
of the semiample torus-invariant $\Q$-Cartier divisor 
$D^{\sigma}$ on the $k$-dimensional variety $X_{\sigma}$. 
It remains to note that the  
number $\left[D^{\sigma}\right]^k$ 
equals $v\left(\Delta_D^{\sigma} \right)$
(cf. \cite[Section 13.4]{CLS11}).
\end{proof}

Using Theorem \ref{ZkCn-k} and Corollary \ref{esing2} respectively 
\ref{corcompl2}, we derive combinatorial formulas for the
stringy Euler number of generic hypersurfaces and 
complete intersections in toric varieties:

\begin{coro} \label{ecomb}
Let $X$ be a $d$-dimensional 
projective $\Q$-Gorenstein toric variety   
associated with a fan $\Sigma$
and $D$ 
a semiample torus-invariant Cartier divisor on $X$ together with
the corresponding lattice polytope  $\Delta_{D}$ . Denote by $Z \subseteq X$ a 
generic semiample Cartier divisor such that  $\left[Z\right] = 
\left[ D\right]$. Then the stringy Euler number of $Z$ is
\begin{align*}
e_{str}(Z)=c_{d-1}^{str} (Z) = \sum_{k=0}^{d-1}
(-1)^{k}\sum_{\sigma \in \Sigma(d-1-k)}  v(\sigma) 
\cdot v \left(\Delta_D^{\sigma}\right),
\end{align*}
where $\Delta_D^{\sigma}$ is a face of $\Delta_{D}$ 
corresponding to a cone $\sigma \in \Sigma$.
If $\left[Z\right]= c_1(X)$, 
the formula simplifies to 
\begin{align*}
e_{str}(Z)=c_{d-1}^{str} (Z) = \sum_{k=0}^{d-3}
(-1)^{k}\sum_{\sigma \in \Sigma(d-1-k)}  v(\sigma) 
\cdot v \left(\Delta_D^{\sigma}\right).
\end{align*}
\end{coro}

\begin{coro} \label{ecomb2}
Let $X$ be a $d$-dimensional 
projective $\Q$-Gorenstein toric variety 
associated with a fan $\Sigma$ 
and $D$
a semiample torus-invariant Cartier divisor on $X$ together with
the corresponding lattice polytope  $\Delta_{D}$.
 Denote by $Z_1, \ldots, Z_r \subseteq X$   
generic semiample Cartier divisors such that  $\left[Z_1\right] = 
\ldots = \left[Z_r\right] =
\left[ D\right]$.
Then the stringy Euler number of the complete intersection $Z_1\cap \ldots \cap Z_r$ is
\begin{align*}
c_{d-r}^{str}\left(Z_1\cap \ldots \cap Z_r\right) 
= \sum_{k=0}^{d-r} (-1)^k \binom{k+r-1}{r-1} 
\sum_{\sigma \in \Sigma(d-r-k)}  v(\sigma) 
\cdot v \left(\Delta_D^{\sigma}\right),
\end{align*}
where $\Delta_D^{\sigma}$ is a face of  $\Delta_{D}$ as above.
\end{coro}

One can generalize Theorem \ref{ZkCn-k} and combinatorially 
compute intersection numbers
$\left[D_1\right]. \;  \ldots \; . \left[D_k\right] . 
c_{d-k}^{str}(X)$, where 
$D_1,  \ldots, D_k$
are different semiample torus-invariant $\Q$-Cartier divisors on $X$. 
For this purpose, we use mixed volumes of faces of some convex 
rational polytopes.

\begin{theo} 
Let $X$ be a $d$-dimensional 
projective $\Q$-Gorenstein toric variety 
associated with a fan $\Sigma$ and
$D_1, \ldots, D_k$ semiample torus-invariant 
$\Q$-Cartier divisors on $X$. 
Then
\begin{align*}
\left[D_1\right]. \; \ldots \; . 
\left[D_k\right].c_{d-k}^{str}(X) = \sum_{\sigma \in \Sigma(d-k)}v(\sigma) 
\cdot v\left(\Delta_{D_1}^{\sigma}, \ldots, \Delta_{D_k}^{\sigma}  \right),
\end{align*}
where
$\Delta_{D_i}^{\sigma}$ is a
face of $\Delta_{D_i}$ $(1 \leq i \leq k)$ corresponding to 
a cone $\sigma \in \Sigma$ 
and $v\left(\Delta_{D_1}^{\sigma}, \ldots, \Delta_{D_k}^{\sigma}
\right)$ denotes the mixed volume of the polytopes 
$\Delta_{D_1}^{\sigma}, \ldots, \Delta_{D_k}^{\sigma}$
with respect to the sublattice $M(\sigma) \subseteq M$.
\end{theo} 

\begin{proof}
Let $\sigma \in  \Sigma(d-k)$ be a $(d-k)$-dimensional cone. Then 
we restrict the semiample torus-invariant 
$\Q$-Cartier divisors $D_1, \ldots, D_k$ to 
the corresponding projective $k$-dimensional toric subvariety 
 $X_\sigma$ of $X$ and obtain $k$ semiample torus-invariant $\Q$-Cartier 
divisors $D_1^{\sigma}, \ldots, D_k^\sigma$ on $X_\sigma$. 
It remains to 
apply  Corollary \ref{cortotalchern} and  the formula in  
\cite[Section 5.4]{Ful93} that claims that
the intersection number 
\begin{align*}
\left[D_1\right]. \; \ldots \; . 
\left[D_k\right] . \left[X_{\sigma}\right] = \left[D_1^{\sigma}\right]. \; \ldots \; . 
\left[D_k^{\sigma}\right]
\end{align*}
can be computed as the  mixed volume
$v\left(\Delta_{D_1}^{\sigma}, \ldots, \Delta_{D_k}^{\sigma}  
\right)$ of the polytopes $ 
\Delta_{D_i}^{\sigma}$. 
\end{proof}

\section{Stringy Libgober-Wood identity 
for toric varieties} \label{section4}
The identity
\begin{align*}
\frac{d^2}{d u^2} E\left(V;u,1\right)\Big\vert_{u=1} =
\frac{3d^2-5d}{12} c_d(V) + \frac{1}{6} c_1(V).c_{d-1}(V)
\end{align*}
has been proved by Libgober and Wood \cite{LW90} for arbitrary 
smooth $d$-dimensional  
projective varieties $V$. 
This identity is equivalent to 
\begin{align*} 
\sum_{0 \leq p,q \leq d} (-1)^{p+q} h^{p,q} (V) 
\left(p-\frac{d}{2} \right)^{2}= \frac{d}{12} c_d(V) 
+ \frac{1}{6} c_1(V).c_{d-1}(V)
\end{align*}
and so the intersection number $c_1(V).c_{d-1}(V)$
can be expressed via the Hodge numbers $h^{p,q} (V)$ of $V$ \cite{Bor97}.

There exists a stringy version of the Libgober-Wood identity 
\begin{align} \label{eq1.2} 
\frac{d^2}{d u^2} E_{str}\left(X;u,1\right)\Big\vert_{u=1} =
\frac{3d^2-5d}{12} c_d^{str}(X) + \frac{1}{6} c_1(X).c_{d-1}^{str}(X),
\end{align}
which holds for any  $d$-dimensional projective variety $X$ 
with at worst log-terminal singularities 
\cite[Theorem 3.8]{Bat00} (cf. Equation \eqref{eq1}).

Moreover, if the singularities of $X$ are at worst canonical Gorenstein and 
the stringy $E$-function $E_{str}(X; u, v)$ is
a polynomial $\sum_{p,q} \psi_{p,q} u^{p} v^{q}$, then one can define 
the {\em stringy Hodge numbers} $h_{str}^{p,q}(X)$ of $X$ \cite{Bat98, Bat00}
as 
\begin{align*}
h_{str}^{p,q}(X) \defeq (-1)^{p+q} \psi_{p,q}.
\end{align*}
In this case, the stringy Libgober-Wood identity can be equivalently 
reformulated \cite[Corollary 3.10]{Bat00} as
\begin{align} \label{eq2}
\sum_{0 \leq p,q \leq d} (-1)^{p+q} h_{str}^{p,q} \left( X \right) \left(p-\frac{d}{2} \right)^{2}
= \frac{d}{12} c^{str}_d (X) + \frac{1}{6} c_1(X).c_{d-1}^{str}(X).
\end{align}

In this section, we are interested in  a combinatorial interpretation
 of these stringy Libgober-Wood identities \eqref{eq1.2}, \eqref{eq2} 
for arbitrary $d$-dimensional 
projective $\Q$-Go\-ren\-stein toric varieties $X$
associated with a fan $\Sigma$. Let $q_X$ be the smallest
positive integer such that 
$q_XK_X$ is a Cartier divisor. The number  $q_X$ 
is called {\em Gorenstein index} of $X$.
Note that the stringy $E$-function  $ E_{str}\left(X;u,v\right)$ 
of such a toric variety $X$  
can be computed combinatorially as
\begin{align} \label{theo4.3}
E_{str}\left(X;u,v\right) 
= (uv-1)^d \sum_{\sigma \in \Sigma} \sum_{n \in \sigma^{\circ} \cap N} (uv)^{\kappa (n)},
\end{align}
where $\kappa$ is the $\Sigma$-piecewise linear function 
corresponding to 
the anticanonical divisor of $X$  
and $\sigma^{\circ}$ is the relative interior of a cone $\sigma \in \Sigma$
\cite[Theorem 4.3]{Bat98}. We remark that 
$\kappa$ has value $-1$ on every primitive lattice generator of a $1$-dimensional 
cone $\sigma \in \Sigma(1)$ and that 
the value ${\kappa (n)}$ $(n \in N)$ 
belongs to $\frac{1}{q_X} \Z$. 

First, we show that the stringy $E$-function 
$E_{str}\left(X;u,v\right)$ is a   
polynomial with nonnegative integral coefficients $\psi_\alpha(\Sigma)$ 
in nonnegative  rational powers $\alpha \in [0,d] \cap \frac{1}{q_X} \Z
$ of $uv$:

\begin{prop} \label{theo4.3new} 
Let $X$ be a $d$-dimensional projective $\Q$-Gorenstein
toric variety of Gorenstein index $q_X$ 
associated with a fan $\Sigma$ in $N_{\R}$ 
and $\Sigma'$ a simplicial subdivision of the 
fan $\Sigma$ such that $\Sigma'(1) =\Sigma(1)$. 
For any cone  $\sigma \in \Sigma'$, we  
denote by $\square_{\sigma}^{\circ}$ 
the relative interior of the 
parallelepiped 
$\square_{\sigma}$
spanned by the  primitive lattice generators of the cone $\sigma$. 
Then the stringy $E$-function can be computed as a 
finite sum 
\begin{align*} 
E_{str}\left(X;u,v\right) = \sum_{\sigma \in \Sigma'} (uv-1)^{d-\dim(\sigma)} 
\sum_{n' \in \square_{\sigma}^{\circ} \cap N} (uv)^{\dim(\sigma)+\kappa (n')}.
\end{align*}
Moreover,  the stringy $E$-function can be written as a finite sum
\begin{align*}
E_{str}\left(X;u,v\right)= \sum_{\alpha \in  [0,d] \cap \frac{1}{q_X} \Z} 
\psi_{\alpha}(\Sigma) (uv)^{\alpha},
\end{align*}
where the coefficients $\psi_{\alpha}(\Sigma)$ are 
nonnegative integers satisfying the conditions 
$\psi_{0}(\Sigma) = \psi_{d}(\Sigma) =1$ and 
$\psi_{\alpha}(\Sigma) = \psi_{d-\alpha}(\Sigma)$ for all 
$\alpha \in  [0,d] \cap \frac{1}{q_X} \Z$. 
\end{prop}

\begin{proof} 
Any $s$-dimensional simplicial cone $\sigma$ 
of $\Sigma'(s)$ is generated by $s$ linearly independent primitive lattice 
vectors $u_1,\ldots, u_s$. Therefore, any lattice point $n \in 
{\sigma}^{\circ} \cap N$ 
has a unique representation as a sum $n= n' + n''$, where 
$n'  = \sum_{i=1}^s \lambda_i u_i \in \square_{\sigma}^{\circ} \cap N$ 
$(0 \leq \lambda_i \leq 1)$ and $n''$ is  a  
linear combination $n'' =\sum_{i=1}^s k_i u_i$ with nonnegative 
integral coefficients $k_i$. Therefore, 
one has 
\begin{align*}
(uv-1)^s \! \! \! \! \sum_{n \in \sigma^{\circ} \cap N} \! \! (uv)^{\kappa (n)}
&= (uv-1)^s \! \! \! \! \sum_{n' \in \square_{\sigma}^{\circ} \cap N} \! \! 
(uv)^{\kappa (n')} \prod_{i=1}^s \bigg( \sum_{k_i \in \Z_{\geq 0}} (uv)^{-k_i} \bigg)  \\
&= (uv-1)^s \! \! \! \!\sum_{n' \in \square_{\sigma}^{\circ} \cap N} \! \!
(uv)^{\kappa (n')} \cdot \bigg( \frac{1}{1-(uv)^{-1}} \bigg)^s 
= \! \! \! \! \sum_{n' \in \square_{\sigma}^{\circ} \cap N}\! \! (uv)^{s+ \kappa (n')}
\end{align*}
and the first statement of Proposition \ref{theo4.3new} 
follows from Equation \eqref{theo4.3}. 
Since $\kappa$ has value $-1$ on every primitive 
lattice generator $u_i$ and $q_X \cdot \kappa(n) \in \Z$ for all $n \in N$,  
 we obtain that  $s + \kappa(n')= 
s -  \sum_{i=1}^s \lambda_i$ is a nonnegative rational number in
 $\frac{1}{q_X} \Z_{\geq 0}$. Therefore,  
$E_{str}\left(X;u,v\right)$ can be written as 
a finite sum $E_{str}\left(X;u,v\right) =
\sum_{\alpha} \psi_{\alpha}(\Sigma) (uv)^\alpha$ for 
some integral coefficients  $\psi_{\alpha}(\Sigma) $ and some 
nonnegative rational numbers $\alpha$ in $\frac{1}{q_X} \Z_{\geq 0}$.  
The  Poincar\'{e} duality \cite[Theorem 3.7]{Bat98} 
for the stringy $E$-function $$E_{str}\left(X;u,v\right) 
= (uv)^d E_{str}\left(X;u^{-1} ,v^{-1}  \right)$$ 
delivers the equalities 
$\psi_{\alpha}(\Sigma) = \psi_{d-\alpha}(\Sigma)$. This implies 
 $\alpha \leq d$ as soon as  $\psi_{\alpha}(\Sigma) \neq 0$. 
Therefore, we obtain 
$$E_{str}\left(X;u,v\right) =
\sum_{\alpha   \in [0, d] \cap 
\frac{1}{q_X} \Z} \psi_{\alpha}(\Sigma) (uv)^\alpha. $$
The nonnegativity of the coefficients  $\psi_{\alpha}(\Sigma)$ can be 
shown using an interpretation of the coefficients 
$\psi_{\alpha}(\Sigma)$ as dimensions 
of graded homogenous components of a graded artinian ring $R$ obtained
as a quotient of a graded Cohen-Macaulay ring $S$ by a regular 
sequence of homogeneous elements 
(cf. \cite[Theorem 2.11]{Bat93}). 
\end{proof} 

\begin{coro}\label{2-dime}
Let $X$ be a $2$-dimensional projective $\Q$-Gorenstein
toric variety associated with a fan $\Sigma$ in $N_{\R}$. Then
\begin{align*} 
E_{str}\left(X;u,v\right) =(uv-1)^2 +  
\sum_{n \in N \atop  \kappa(n)=-1} 
 uv + \sum_{n \in N \atop -1 < \kappa(n) < 0} 
\left( (uv)^{2+\kappa(n)} + (uv)^{-\kappa(n)} \right).
\end{align*}
\end{coro}

\begin{proof}
We do not need a subdivision $\Sigma'$ of the fan 
$\Sigma$ because every cone $\sigma \in \Sigma$ already is simplicial. 
Therefore, we set $\Sigma' = \Sigma$. 
By  Proposition \ref{theo4.3new}, we obtain 
\begin{align*}
E_{str} \left(X;u,v\right) 
&= \sum_{\sigma \in \Sigma } (uv-1)^{2-\dim(\sigma)} 
\sum_{n \in \square_{\sigma}^{\circ} \cap N} (uv)^{\dim(\sigma)+\kappa (n)} \\
&=(uv-1)^2 + \sum_{\sigma \in \Sigma(1)} (uv-1) + 
\sum_{\sigma \in \Sigma(2)} \sum_{n \in \square_{\sigma}^{\circ} \cap N} (uv)^{2+\kappa(n)}.
\end{align*}

For any $2$-dimensional cone $\sigma \in \Sigma(2)$, the set 
$\{ x \in N_\R \; : \; \kappa(x) =-1\}$ divides the parallelogram
$\square_{\sigma}$ into two isomorphic lattice triangles 
$\triangle^\sigma_{\leq -1}$ and  $\triangle^\sigma_{\geq -1}$. Let 
$u_1, u_2$ be the primitive lattice generators of $\sigma$.  
We can write 
every lattice point $n \in \square_\sigma$ as a linear combination
$n = \lambda_1 u_1 + \lambda_2 u_2$ with rational coefficients 
$\lambda_1, \lambda_2 \in [0,1]$. A lattice point $n \in \square_{\sigma}$ 
belongs to the 
triangle $\triangle^\sigma_{\geq -1}$ if and only if the lattice point
$n^*:= u_1 +  u_2 -n$ belongs to the triangle $\triangle^\sigma_{\leq -1}$. 
Since the boundary of the lattice  parallelogram
$\square_{\sigma}$ has no lattice points except vertices, we can use the 
bijection 
 $n \leftrightarrow n^*$ together with the equation 
$\kappa(n) + \kappa(n^*) = -2$ to obtain 
\begin{align*} 
\sum_{n \in \square_{\sigma}^{\circ} \cap N} (uv)^{2+\kappa(n)} = 
1 + \sum_{n \in \sigma^\circ \cap N  \atop  \kappa(n)=-1} 
 uv  + \sum_{n \in \sigma \cap N \atop -1 < \kappa(n)<0} 
\left( (uv)^{2+\kappa(n)} + (uv)^{-\kappa(n)} \right). 
\end{align*}
It remains to apply the equalities $|\Sigma(1) | =  |\Sigma(2) |$
and 
\begin{align*} 
\sum_{\sigma \in \Sigma(1)}  uv + \sum_{\sigma \in \Sigma(2)} 
\sum_{n \in \sigma^\circ \cap N  \atop  \kappa(n)=-1} uv = 
\sum_{n \in N \atop  \kappa(n)=-1} uv . 
\end{align*}
\end{proof}

The equality $E_{str}\left(X;u,v\right) = 
\sum_{\alpha} \psi_{\alpha}\left(\Sigma\right)  (uv)^{\alpha}$  
in Proposition \ref{theo4.3new} suggests that
the nonnegative integral 
coefficients  $\psi_{\alpha}\left(\Sigma\right)$
may be interpreted 
as {\em generalized stringy Hodge numbers} $h^{\alpha, \alpha}_{str}(X)$ of the 
toric variety $X$ for some 
rational numbers $\alpha \in [0, d] \cap 
\frac{1}{q_X} \Z$.  

The following theorem presents a
combinatorial interpretation for the second version 
of  the stringy Libgober-Wood identity \eqref{eq2} 
using the generalized stringy Hodge numbers of the toric variety $X$.

\begin{theo} \label{refstrLW}
Let $X$ be a $d$-dimensional projective $\Q$-Gorenstein
toric variety of Gorenstein index $q_X$  associated with a fan $\Sigma$ 
and $-K_X = \sum_{\rho \in \Sigma(1)} D_{\rho}$ 
the anticanonical torus-invariant $\Q$-Cartier 
divisor on $X$.
Then the stringy
Libgober-Wood identity 
is equivalent to
\begin{align*}
\sum_{\alpha \in [0, d] \cap 
\frac{1}{q_X}\Z} \psi_{\alpha} \left(\Sigma\right) 
\left(\alpha - \frac{d}{2} \right)^2 = \frac{d}{12} v(\Sigma) 
+ \frac{1}{6} \sum_{\sigma \in \Sigma(d-1)}v(\sigma) 
\cdot l_{-K_X}(\sigma),
\end{align*}
where
$\psi_{\alpha} \left(\Sigma\right) $ are nonnegative integers 
as above,
$v\left(\Sigma\right) \defeq \sum_{\sigma \in \Sigma(d)} v (\sigma)$, and
$l_{-K_X}(\sigma) \in \Q$ is the intersection number 
$[-K_X] .[X_\sigma]$ 
(cf. Proposition \ref{C1Cn-1}).
If $-K_X$ is semiample, then
\begin{align*}
\sum_{\alpha \in [0, d] \cap 
\frac{1}{q_X}\Z} \psi_{\alpha} \left(\Sigma\right) 
\left(\alpha - \frac{d}{2} \right)^2 = \frac{d}{12} v(\Sigma) 
+ \frac{1}{6} \sum_{\sigma \in \Sigma(d-1)}v(\sigma) 
\cdot v\left(\Delta_{-K_X}^{\sigma} \right) ,
\end{align*}
where 
$\Delta_{-K_X}^{\sigma}$ is a face of the rational 
polytope $\Delta_{-K_X}$ corresponding to a cone $\sigma \in \Sigma$
(cf. Theorem \ref{ZkCn-k}).
\end{theo}

\begin{proof}
Using the equality  $E_{str}\left(X;u,v\right) = 
\sum_{\alpha} \psi_{\alpha}\left(\Sigma\right)  (uv)^{\alpha}$
from Proposition \ref{theo4.3new},
we obtain 
\begin{align*}
\frac{d^2}{d u^2} E_{str}\left(X;u,1\right) \vert_{u=1}
=  \sum _{\alpha} \alpha \cdot (\alpha-1) \psi_{\alpha}\left(\Sigma\right),
\end{align*}
i.e., the stringy Libgober-Wood identity \eqref{eq1.2} is given as
\begin{align*}
\sum_{\alpha} \left(\alpha^2 - \alpha \right) \psi_{\alpha}\left(\Sigma\right)
= \frac{3d^2-5d}{12} c_d^{str}(X) + \frac{1}{6} c_1(X).c_{d-1}^{str}(X) .
\end{align*}
Applying $ \sum_{\alpha} \alpha \psi_{\alpha}  \left(\Sigma\right) 
= \frac{d}{d u}  E_{str}\left(X;u,1\right) \vert_{u=1} = \frac{d}{2}c_d^{str}(X)$ 
\cite[Proposition 3.4]{Bat00} 
a short calculation yields 
\begin{align*}
\sum_{\alpha} \alpha^2 \psi_{\alpha} \left(\Sigma\right)
&=  \big( \frac{d}{12} + \frac{d^2}{4}\big) c_d^{str}(X) + \frac{1}{6} c_1(X).c_{d-1}^{str}(X)
\end{align*}
and implies
\begin{align*}
\sum_{\alpha} \psi_{\alpha}\left(\Sigma\right) \left(\alpha - \frac{d}{2} \right)^2 
&=   \sum_{\alpha} \alpha^2 \psi_{\alpha} \left(\Sigma\right)
- d \sum_{\alpha} \alpha \psi_{\alpha}\left(\Sigma\right) 
+ \frac{d^2}{4} \sum_{\alpha}  \psi_{\alpha}\left(\Sigma\right) \\
&=  \frac{d}{12} c_d^{str}(X)+ \frac{1}{6} c_1(X).c_{d-1}^{str}(X)
\end{align*}
because $ \sum_{\alpha}  \psi_{\alpha} \left(\Sigma\right)= 
 E_{str}\left(X;u,1\right) \vert_{u=1} = c_d^{str}(X)$ 
\cite[Definition 2.1]{Bat00}. 
To finish, it remains to note that 
$c_d^{str}(X) = v(\Sigma)$ by Corollary \ref{cortotalchern} 
(cf. \cite[Proposition 4.10]{Bat98}),
\begin{align*}
c_1(X).c_{d-1}^{str}(X) 
= \left[-K_X\right].c_{d-1}^{str}(X)
= \sum_{\sigma \in \Sigma(d-1)}v(\sigma) 
\cdot l_{-K_X}(\sigma)
\end{align*}
by Proposition \ref{C1Cn-1}, and 
\begin{align*}
c_1(X).c_{d-1}^{str}(X) 
= \left[-K_X\right].c_{d-1}^{str}(X)
= \sum_{\sigma \in \Sigma(d-1)}v(\sigma) 
\cdot v\left(\Delta_{-K_X}^{\sigma} \right) 
\end{align*}
by Theorem \ref{ZkCn-k}
if $-K_X$ is semiample.
\end{proof}

We formulate one more combinatorial version of the stringy
Libgober-Wood identity containing only intrinsic informations coming from
the associated fan $\Sigma$ of the toric variety $X$. To achieve this, 
we describe the left side of the 
stringy Libgober-Wood identity \eqref{eq1.2} in pure combinatorial terms using Proposition 
\ref{theo4.3new}.

\begin{theo} \label{theo3.8new}
Let $X$ be a $d$-dimensional projective $\Q$-Gorenstein
toric variety associated with a fan $\Sigma$ in $N_{\R}$
and $-K_X = \sum_{\rho \in \Sigma(1)} D_{\rho}$ the anticanonical 
torus-invariant $\Q$-Cartier divisor on $X$.
Then the stringy
Libgober-Wood identity 
is equivalent to
\begin{align*} 
2 &\cdot  \sum_{\sigma \in \Sigma'(d-2)} 
\left\vert  \square_{\sigma}^{\circ} \cap N \right\vert 
+2  \cdot  \sum_{\sigma \in \Sigma'(d-1)}  \sum_{n' \in \square_{\sigma}^{\circ} \cap N}  
\left( d+\kappa(n') -1 \right) \\
&+  \! \!  \sum_{\sigma \in \Sigma'(d)}  \sum_{n' \in \square_{\sigma}^{\circ} \cap N}  \! \! \! \! 
\left( d+\kappa(n') \right) \left( d+\kappa(n')-1 \right) 
= \frac{(3d^2-5d)}{12}    v\left(\Sigma\right) + 
\frac{1}{6}    c_1(X).c_{d-1}^{str}(X),
\end{align*}
where $\Sigma'$ is a simplicial
subdivision of the fan $\Sigma$ such that $\Sigma'(1) = \Sigma(1)$,
$\square_{\sigma}^{\circ}$ is the relative interior of the parallelepiped
$\square_{\sigma}$
spanned by the primitive lattice
generators of a cone $\sigma \in \Sigma$, $\kappa$ is the 
$\Sigma$-piecewise linear function corresponding to $-K_X$, and
$v(\Sigma) = \sum_{\sigma \in \Sigma(d)} v(\sigma)$. The rational
number $c_1(X).c_{d-1}^{str}(X) \in \Q$ is computable as
\begin{align*}
c_1(X).c_{d-1}^{str}(X)=\sum_{\sigma \in \Sigma(d-1)} v(\sigma) \cdot 
l_{-K_X}(\sigma),
\end{align*}
where $l_{-K_X}$ is the intersection number 
$[-K_X] .[X_\sigma]$ 
(cf. Proposition \ref{C1Cn-1}).
If $-K_X$ is semiample, then
\begin{align*} 
 c_1(X).c_{d-1}^{str}(X)= \sum_{\sigma \in \Sigma(d-1)} v(\sigma) \cdot 
 v\left(\Delta_{-K_X}^{\sigma}\right),
\end{align*}
where 
$\Delta_{-K_X}^{\sigma}$ is a face of the rational polytope $\Delta_{-K_X}$ 
(cf. Theorem \ref{ZkCn-k}).
\end{theo}

\begin{proof}
We derivate one summand of 
the stringy $E$-function 
\begin{align*} 
E_{str}\left(X;u,v\right) = \sum_{\sigma \in \Sigma'} (uv-1)^{d-\dim(\sigma)} 
\sum_{n' \in \square_{\sigma}^{\circ} \cap N} (uv)^{\dim(\sigma)+\kappa (n')}.
\end{align*}
from Proposition \ref{theo4.3new}
inserted $v=1$ twice and get
\begin{align*} 
\frac{d^2}{d u^2}  (u-1)^{d-s} 
\sum_{n' \in \square_{\sigma}^{\circ} \cap N} & u^{s+\kappa (n')} 
= (d-s)(d-1-s) (u-1)^{d-2-s} \sum_{n' \in \square_{\sigma}^{\circ} \cap N} u^{s+\kappa(n')}\\
&+ 2 \cdot (d-s)(u-1)^{d-1-s} \sum_{n' \in \square_{\sigma}^{\circ} \cap N} (s+\kappa(n')) 
u^{s+\kappa(n')-1}\\
&+ (u-1)^{d-s} \sum_{n' \in \square_{\sigma}^{\circ} \cap N} 
(s+\kappa(n'))(s+\kappa(n')-1) u^{s+\kappa(n')-2},
\end{align*}
where $\sigma$ is any $s$-dimensional cone of $\Sigma'(s)$. 
Inserting $u=1$ the relevant cones of $\Sigma'$ are
these of dimension $d$, $d-1$, and $d-2$, i.e.,
\begin{align*} 
\frac{d^2}{d u^2} E_{str}\left(X;u,1\right)\Big\vert_{u=1} 
&= 2 \cdot \! \!  \sum_{\sigma \in \Sigma'(d-2)}  \! \! 
\left\vert  \square_{\sigma}^{\circ} \cap N \right\vert 
+ 2   \cdot \! \! \sum_{\sigma \in \Sigma'(d-1)}  \sum_{n' \in \square_{\sigma}^{\circ} \cap N}  
\left( d+ \kappa(n')-1 \right) \\
&+ \sum_{\sigma \in \Sigma'(d)}  \sum_{n' \in \square_{\sigma}^{\circ} \cap N}  
\left( d+ \kappa(n') \right) \left( d+ \kappa(n')-1 \right).
\end{align*}
By Equation \eqref{eq1.2}, we obtain the equality
\begin{align*} 
2 &\cdot  \sum_{\sigma \in \Sigma'(d-2)} 
\left\vert  \square_{\sigma}^{\circ} \cap N \right\vert 
+2  \cdot  \sum_{\sigma \in \Sigma'(d-1)}  \sum_{n' \in \square_{\sigma}^{\circ} \cap N}  
\left( d+\kappa(n') -1 \right) \\
&+  \! \!  \sum_{\sigma \in \Sigma'(d)}  \sum_{n' \in \square_{\sigma}^{\circ} \cap N}  \! \! \! \! 
\left( d+\kappa(n') \right) \left( d+\kappa(n')-1 \right) 
= \frac{(3d^2-5d)}{12}    v\left(\Sigma\right) + 
\frac{1}{6}    c_1(X).c_{d-1}^{str}(X)
\end{align*}
because $c_d^{str}(X) = v(\Sigma)$ by Corollary \ref{cortotalchern} 
(cf. \cite[Proposition 4.10]{Bat98}). 
Furthermore,
\begin{align*}
c_1(X).c_{d-1}^{str}(X) 
= \left[-K_X\right].c_{d-1}^{str}(X)
= \sum_{\sigma \in \Sigma(d-1)}v(\sigma) 
\cdot l_{-K_X}(\sigma)
\end{align*}
by Proposition \ref{C1Cn-1} and 
\begin{align*}
c_1(X).c_{d-1}^{str}(X) 
= \left[-K_X\right].c_{d-1}^{str}(X)
= \sum_{\sigma \in \Sigma(d-1)}v(\sigma) 
\cdot v\left(\Delta_{-K_X}^{\sigma} \right) 
\end{align*}
by Theorem \ref{ZkCn-k}
if $-K_X$ is semiample.
\end{proof}

Recall that a normal projective surface  
is called {\em log del Pezzo surface} if it has 
at worst log-terminal singularities 
and if its anticanonical divisor is an ample $\Q$-Cartier divisor. 
Toric log del Pezzo surfaces one-to-one correspond to   
convex lattice polygons $\Delta \subseteq N_\R$ containing  
the origin in its interior such that  the vertices of $\Delta$ 
are primitive lattice points in $N$. These polygons $
\Delta$  are  called {\em LDP-polygons} \cite{KKN10}. 
The fan $\Sigma$ defining a {toric log del Pezzo surface} 
$X$ consists of cones over
faces of $\Delta$. In particular,  any LDP-polygon $\Delta$ 
is the convex hull of all primitive lattice generators of $1$-dimensional 
cones of $\Sigma(1)$. We remark that in general the vertices of the dual 
polygon $\Delta^* \subseteq M_\R$ are not lattice points in $M$. 

We propose  a new combinatorial identity  that
is equivalent to the stringy Libgober-Wood identity \eqref{eq1.2} and  
relates the number $12$ to LDP-polygons $\Delta$:

\begin{coro} \label{pezzo12} 
Let $X$ be a toric log del Pezzo surface defined by a fan $\Sigma$ in $N_{\R}$
together with the corresponding LDP-polygon $\Delta \subseteq N_{\R}$. Then 
 \begin{align*}
v\left(\Delta \right) +  v \left(\Delta^* \right)= 12 \sum_{n \in \Delta \cap N} \left( \kappa(n) + 1 \right)^2,
\end{align*}
where $\kappa$ the $\Sigma$-piecewise linear function corresponding to the anticanonical divisor of $X$.
In particular, one always has $v\left(\Delta \right) + 
 v \left(\Delta^* \right) \geq 12$ 
and equality holds if and only if 
$\Delta$ is a reflexive polygon. 
\end{coro}

\begin{proof} 
We use the formula for the stringy $E$-function from Corollary \ref{2-dime} and obtain
 \begin{align*} 
E_{str}\left(X;u,1\right) =(u-1)^2 +  
\sum_{n \in N \atop  \kappa(n)=-1} 
 u + \sum_{n \in N \atop -1 < \kappa(n) < 0} 
\left( u^{2+\kappa(n)} + u^{-\kappa(n)} \right).
\end{align*}
Therefore, 
\begin{align*} 
\frac{d^2}{d u^2} E_{str}\left(X;u,1\right)\Big\vert_{u=1} &= 
2 + \! \! \sum_{0 \neq n \in   \Delta^\circ \cap N} \! \! \! \! \left( (2 + \kappa(n) ) 
(1 + \kappa(n)) + 
(- \kappa(n) ) (-\kappa(n)-1) \right) \\
 &= 2 \sum_{ n \in   \Delta^\circ \cap N} ( \kappa(n) +1)^2 =  
2 \sum_{ n \in   \Delta \cap N} ( \kappa(n) +1)^2 ,  
\end{align*} 
where $\Delta^\circ$ denotes the interior of the polygon $\Delta$.  
By Equation \eqref{eq1.2}, we get the equality 
\begin{align*}
2 \sum_{ n \in   \Delta \cap N} ( \kappa(n) +1)^2 = 
\frac{1}{6}c_2^{str}(X) + \frac{1}{6} c_1(X)^2 = 
\frac{1}{6} (v(\Delta) +  v(\Delta^*))
\end{align*}
because $c_2^{str}(X) = v(\Sigma)= v(\Delta)$ and  $c_1(X)^2 = v(\Delta^*)$. 
One has 
\begin{align*}
\sum_{ n \in   \Delta \cap N} ( \kappa(n) +1)^2 \geq 1
\end{align*}
because the origin 
is contained in $\Delta$. Equality holds, if and only if the origin is the unique 
interior lattice point $n$ in $\Delta$. 
\end{proof} 

\section{Applications to reflexive and Gorenstein polytopes} \label{section5}

Let $\Delta \subseteq N_{\R}$ be a $d$-dimensional convex 
lattice polytope that contains the origin in its interior.
Denote  by $\Sigma$ a fan in $N_{\R}$ consisting of 
cones over faces of $\Delta$ that defines a normal projective toric 
variety $X$. 
The polytope  $\Delta$ is called reflexive 
if its dual 
$\Delta^{*}= \left\{y \in M_{\R} \middle\vert 
\left<y,x\right> \geq -1 \; \forall x \in \Delta \right\}$ 
is also a lattice polytope.
If $\Delta$ is reflexive, then the associated variety $X$ 
is a 
Gorenstein toric Fano variety (i.e., $q_X=1$). 

We are interested in a  combinatorial identity
for reflexive polytopes $\Delta$   that
is equivalent to the stringy Libgober-Wood identity \eqref{eq2} 
for   Gorenstein toric Fano varieties.
For this purpose, we observe 
that the generalized stringy Hodge numbers 
$\psi_{\alpha}\left(\Sigma \right)$ in  
\begin{align*}
E_{str}(X;u,v) = \sum_{\alpha \in [0,d] \cap \Z} \psi_{\alpha} \left(\Sigma \right) (uv)^{\alpha}
\end{align*}
are equal to the nonnegative 
integral coefficients 
$\psi_d\left(\Delta \right), \ldots, \psi_0\left(\Delta \right)$
in the numerator of the Ehrhart power series
\begin{align} \label{eps}
P_{\Delta}(t) =\frac{\psi_d\left(\Delta \right) t^d + \ldots+\psi_1\left(\Delta \right)t
+\psi_0\left(\Delta \right)}{(1-t)^{d+1}}. 
\end{align}

\begin{lem} \label{hodgepsi}
Let $\Delta \subseteq N_{\R}$ be a $d$-dimensional reflexive polytope and $X$ the
associated Gorenstein toric Fano variety. Then
\begin{align*}
E_{str}(X;u,v) = \psi_d\left(\Delta \right) \left( uv\right)^d + \ldots+\psi_1\left(\Delta \right) \left( uv\right)+\psi_0\left(\Delta \right),
\end{align*}
i.e., $\psi_{\alpha} \left(\Sigma \right) =\psi_{\alpha}  \left(\Delta \right)$ 
for all ${\alpha}  \in [0,d] \cap \Z$,
where $\psi_{\alpha} \left(\Sigma \right)$ and $\psi_{\alpha}  \left(\Delta \right)$ are given as above.
\end{lem}

\begin{proof}
By Equation \eqref{theo4.3}, we have 
\begin{align*}
E_{str}\left(X;u,v\right)
= (uv-1)^d \sum_{n \in  N} \left( uv \right)^{\kappa (n)} 
= (uv-1)^d \sum_{k \geq 0} \sum_{ n \in N \atop \kappa(n)=-k} (uv)^{-k} ,  
\end{align*}
since the  fan $\Sigma$ defining $X$ is complete. 
We note that the number of lattice 
points $n \in N$ such that $\kappa(n) = -k$ 
equals   $ \left\vert k\Delta \cap N \right\vert - 
 \left\vert (k-1)\Delta \cap N \right\vert$.  
Therefore, we get 
\begin{align*}
E_{str}\left(X;u,v\right)
&= (uv-1)^d\left( 1- (uv)^{-1} \right) 
\sum_{k \geq 0}  \left\vert k\Delta \cap N \right\vert  (uv)^{-k}.
\end{align*}
Using  the definition of $P_{\Delta}\left( (uv)^{-1} \right)$ and 
Equation \eqref{eps},  this implies
\begin{align*}
E_{str}\left(X;u,v\right)
&= (uv)^d \left(1-(uv)^{-1}\right)^{d+1}  
 \cdot P_{\Delta}\left( (uv)^{-1} \right) \\
&=  \psi_d\left(\Delta \right)  + 
\ldots+\psi_1\left(\Delta \right)\left( uv \right)^{d-1}+\psi_0\left(\Delta \right) (uv)^{d}  \\
&=  \psi_d\left(\Delta \right) (uv)^{d} + 
\ldots+\psi_1\left(\Delta \right)\left( uv \right)+\psi_0\left(\Delta \right) 
\end{align*}
because $\psi_\alpha \left(\Delta \right) = 
\psi_{d- \alpha}\left(\Delta \right)$ for all $0 \leq \alpha \leq d$ \cite[Theorem 2.11]{Bat93}.
\end{proof}

Let $\Delta \subseteq N_{\R}$ be a $d$-dimensional convex 
lattice polytope that contains the origin in its interior.
If $\theta$ is a face of $\Delta$, then the face
$\theta^{*} = \{y \in \Delta^{*} \vert \left<y,x\right> = -1 
\; \forall x \in \theta \} \subseteq \Delta^{*}$ 
is the dual face to  $\theta$. 
This establishes a one-to-one order-reversing correspondence 
between faces of $\Delta$ and 
faces of $\Delta^{*}$ such that $\dim(\theta) + \dim \left( \theta^{*} \right) = d -1$. 

\begin{theo} \label{3.10propnew}
Let $\Delta \subseteq N_{\R}$ be a $d$-dimensional 
reflexive polytope. Then the stringy Libgober-Wood identity  
for the Gorenstein toric Fano variety $X$ 
corresponding to $\Delta$ is 
equivalent to 
\begin{align*} 
\sum_{\alpha \in [0,d] \cap \Z} \psi_{\alpha} \left(\Delta \right)  \left(\alpha-\frac{d}{2} \right)^{2}
= \frac{d}{12} v\left(\Delta \right)+ \frac{1}{6} 
\sum_{\theta \preceq \Delta \atop  \dim(\theta)=d-2} v(\theta) \cdot v \left(\theta^{*}\right),
\end{align*}
where 
$\psi_{\alpha} \left( \Delta \right)$ are the coefficients 
in the numerator of the Ehrhart power series $P_{\Delta}(t)$.
\end{theo}

\begin{proof}
Using Lemma \ref{hodgepsi} and Theorem \ref{refstrLW}, we obtain
\begin{align*} 
\sum_{\alpha \in [0,d] \cap \Z}  \psi_{\alpha}\left(\Delta \right)   \left(\alpha-\frac{d}{2} \right)^{2} 
&=\frac{d}{12} v \left( \Sigma\right)+ \frac{1}{6} \sum_{\sigma \in \Sigma(d-1)}v(\sigma) 
\cdot v\left(\Delta_{-K_X}^{\sigma} \right),
\end{align*} 
since the anticanonical divisor
$-K_X$ is  ample. 
It remains to use  
$v \left( \Sigma \right)= v\left( \Delta \right)$ 
and 
\begin{align*}
\sum_{\sigma \in \Sigma(d-1)}v(\sigma) 
\cdot v\left(\Delta_{-K_X}^{\sigma} \right) 
=\sum_{\theta \preceq \Delta, \atop  \dim\left(\theta \right)=d-2}
v \left(\theta \right) \cdot 
v \left(\theta^{*}\right), 
\end{align*}
since  $\Delta_{-K_X} = \Delta^*$, $\sigma$ is a cone 
over a face $\theta$ of $\Delta$, 
$\Delta_{-K_X}^{\sigma} = \theta^*$, and 
every facet of $\Delta$ has lattice distance $1$ to the origin.
\end{proof}

The well-known identities for reflexive polytopes of
dimension $2$ and $3$ follow from the above statement:

\begin{coro} \label{212}
Let $\Delta \subseteq N_{\R}$ be a $2$-dimensional reflexive polytope. Then the stringy Libgober-Wood identity is equivalent to 
\begin{align*} 
 v(\Delta) + v\left( \Delta^{*} \right)=12.
\end{align*}
\end{coro}

\begin{proof}
Using Theorem \ref{3.10propnew}, we get
\begin{align*} 
\sum_{\alpha \in [0,2] \cap \Z} \psi_{\alpha} \left(\Delta \right)  \left(\alpha-1 \right)^{2}
= \frac{1}{6} v\left( \Delta \right) + \frac{1}{6} \sum_{\theta \preceq \Delta 
\atop  \dim(\theta)=0} 
v(\theta) \cdot v \left(\theta^{*}\right) .
\end{align*}
Moreover, 
$\sum_{\alpha \in [0,2] \cap \Z} \psi_\alpha\left(\Delta \right)   \left(\alpha -1 \right)^{2}
= 2$
because 
$\psi_0\left(\Delta \right)= \psi_2\left(\Delta \right)=1$.
It remains to apply the equalities 
\begin{align*}
\sum_{\theta \preceq \Delta \atop  \dim(\theta)=0}
 v(\theta) \cdot v \left(\theta^{*}\right)  = 
\sum_{\theta \preceq \Delta \atop  \dim(\theta)=0}
 v \left(\theta^{*}\right) 
=  v\left( \Delta^{*} \right)  
\end{align*}
that hold because $\Delta^*$ is reflexive and 
$v\left(\theta \right) =1$ if  $\dim(\theta)=0$.
\end{proof}

\begin{coro} \label{324}
Let $\Delta \subseteq N_{\R}$ be a $3$-dimensional reflexive polytope. Then  the stringy Libgober-Wood identity is equivalent to 
\begin{align*} 
\sum_{\theta \preceq \Delta \atop  \dim(\theta)=1} 
v\left(\theta\right) \cdot v\left(\theta^{*}\right)=24.
\end{align*}
\end{coro}

\begin{proof}
Theorem \ref{3.10propnew} implies
\begin{align*} 
\sum_{\alpha \in [0,3] \cap \Z} \psi_{\alpha}\left(\Delta \right)  \left(\alpha-\frac{3}{2} \right)^{2}
= \frac{1}{4} v\left( \Delta \right)+ \frac{1}{6}\sum_{\theta \preceq \Delta \atop  \dim(\theta)=1} 
v(\theta) \cdot v \left(\theta^{*}\right).
\end{align*}
The coefficients $\psi_{\alpha}\left(\Delta \right)$ in the numerator of the Ehrhart power series $P_{\Delta}(t)$ are 
$\psi_0\left(\Delta \right)=\psi_3\left(\Delta \right)=1$ and
$\psi_1\left(\Delta \right) = \psi_2\left(\Delta \right)= \left\vert  \Delta \cap N \right\vert -4$ 
\cite[Theorem 2.11]{Bat93},
i.e., 
\begin{align*} 
\sum_{\alpha \in [0,3] \cap \Z} \psi_{\alpha}\left(\Delta \right) \left(\alpha-\frac{3}{2} \right)^{2}= 
\frac{9}{2}  +\frac{1}{2} \psi_1\left(\Delta \right)
\end{align*}
and we conclude 
\begin{align*} 
\sum_{\theta \preceq \Delta \atop  \dim(\theta)=1} v(\theta) \cdot v \left(\theta^{*}\right)
=27  + 3 \big( \psi_1\left(\Delta \right) - \frac{1}{2} v\left( \Delta \right) \big)
=24,
\end{align*}
where the last equality holds because 
$v\left( \Delta \right)
= v(\Sigma) = e_{str}(X)
= \sum_{\alpha \in [0,3] \cap \Z} \psi_{\alpha}\left(\Delta \right)$.
\end{proof}

If the  dimension $d$ of a reflexive polytope 
is greater than  $3$, the identity in Theorem \ref{3.10propnew} 
is not anymore a symmetric equation with respect to 
polar duality between $\Delta$ and $\Delta^{*}$.  
The received identities for reflexive polytopes $\Delta$ 
and $\Delta^* $ of dimension $\geq 4$ are
not equivalent to each other. The latter is easy to see in the 
case $d=4$:   

\begin{coro} \label{424}
Let $\Delta \subseteq N_{\R}$ be a $4$-dimensional reflexive polytope. Then the stringy Libgober-Wood indentity is equivalent to 
\begin{align*} 
12 \cdot \left\vert \partial \Delta \cap N \right\vert= 2 \cdot v\left( \Delta \right) 
+  \sum_{\theta \preceq \Delta \atop  \dim(\theta)=2} v(\theta) \cdot v \left(\theta^{*}\right),
\end{align*}
where $\partial \Delta$ denotes the boundary of $\Delta$ and 
$\left\vert \partial \Delta \cap N \right\vert$ the number of lattice points in $\partial \Delta$.
\end{coro}

\begin{proof}
By Theorem \ref{3.10propnew}, we have  
\begin{align*}
\sum_{\alpha \in [0,4] \cap \Z} \psi_{\alpha}\left(\Delta \right)   \left(\alpha-2 \right)^{2} 
= \frac{1}{3} v\left( \Delta \right) + \frac{1}{6}  
\sum_{\theta \preceq \Delta \atop  \dim(\theta)=2} v(\theta) \cdot v \left(\theta^{*}\right).
\end{align*}
Furthermore, $\psi_0\left(\Delta \right)=\psi_4\left(\Delta \right)=1$ and 
$\psi_1\left(\Delta \right)=\psi_3\left(\Delta \right)= \left\vert  \Delta \cap N \right\vert -5$
 \cite[Theorem 2.11]{Bat93}, i.e., we obtain 
\begin{align*}
\sum_{\alpha \in [0,4] \cap \Z} \psi_{\alpha}\left(\Delta \right)   \left(\alpha-2 \right)^{2} = 
8+2 \cdot \left( \left\vert  \Delta \cap N \right\vert -5 \right).
\end{align*}
It remains to apply 
 $\left\vert \partial \Delta \cap N \right\vert =  \left\vert  \Delta \cap N \right\vert -1 $ 
 because a reflexive polytope $\Delta$ has a single 
interior lattice point.
\end{proof}

One may produce a more ``mirror symmetric'' identity 
for arbitrary $4$-di\-men\-sio\-nal reflexive polytopes 
by summing the equations from Corollary \ref{424}
for $\Delta$ and $\Delta^*$.

\begin{coro} \label{424c}
Let $\Delta \subseteq N_{\R}$ be a $4$-dimensional reflexive polytope. Then
 \begin{align*} 
12 \cdot  \left(\left\vert \partial \Delta \cap N \right\vert+ \left\vert \partial \Delta^* \cap M \right\vert \right) &= 2 \cdot \left( v\left( \Delta \right) + v\left( \Delta^* \right)\right)
+  \sum_{\theta \preceq \Delta \atop  \dim(\theta)=1,2} v(\theta) \cdot v \left(\theta^{*}\right).
\end{align*}
\end{coro}

Let $r$ be a positive integer. A   $d$-dimensional lattice polytope 
$\Delta \subseteq M_{\R}$  is called 
{\em Gorenstein polytope of index} $r$ 
if $r\Delta -m$ is a reflexive polytope for some 
lattice point $m \in M$. 
Note that reflexive polytopes are Gorenstein polytopes of 
index $r=1$. 
Denote by $\Sigma$ the normal fan in $N_{\R}$ of the polytope 
$\Delta$ (or, equivalently, of $r\Delta$). The fan $\Sigma$ 
defines a Gorenstein toric Fano variety $X$ such that its anticanonical 
class $c_1(X)$ is divisible by $r$ in $\Pic(X)$.

There exists a duality for Gorenstein polytopes that generalizes
the polar duality for reflexive polytopes. 
For this purpose, we associate to every
$d$-dimensional  Gorenstein polytope $\Delta \subseteq M_{\R}$ of index $r$
the $(d+1)$-dimensional cone
\begin{align*}
C_{\Delta} \defeq \left\{ (y,\lambda) \in   
M_{\R} \oplus \R\; \middle\vert \; 
y \in \lambda \Delta \right\} \subseteq M_{\R} \oplus \R.
\end{align*}
The dual cone $C_{\Delta}^{\vee}  \subseteq N_{\R} \oplus \R$ is defined as 
$$C_{\Delta}^{\vee} \defeq \{ (x, \mu) \in N_{\R} \oplus \R
\; \vert \;  \left<y,x \right> + \lambda \mu \geq 0 \; \forall (y, \lambda)
 \in C_{\Delta} \}$$ 
and the $l$-th slice $C_{\Delta}^{\vee}(l)$ of $C_{\Delta}^{\vee}$ is defined as
the lattice polytope 
\begin{align*}
C_{\Delta}^{\vee}(l) \defeq C_{\Delta}^{\vee} \cap 
\{ (x, \mu) \in N_{\R} \oplus \R \; \vert \;  \left<m,x \right> + r\mu = l \}\subseteq N_{\R} \oplus \R.
\end{align*}
The lattice polytope $\Delta^* \defeq C_{\Delta}^{\vee}(1)$ is again
 a Gorenstein polytope of index $r$ and is called
{\em dual Gorenstein polytope} to $\Delta$.
The duality between two $(d+1)$-dimensional cones $C_{\Delta}$ and 
$C_{\Delta}^{\vee}$
 establishes a one-to-one order-reversing 
correspondence between faces of $C_{\Delta}$
and $C_{\Delta}^{\vee}$ that induces a duality between faces of the 
Gorenstein polytopes  $\Delta$
and $\Delta^* = C_{\Delta}^{\vee}(1)$.
It is important to note that 
the reflexive polytope $(r\Delta)^*$ and the Gorenstein 
polytope $\Delta^*$ are not only naturally 
combinatorially isomorphic, but this isomorphism also induces
isomorphisms between proper 
faces of $(r\Delta)^*$ and $\Delta^*$ considered as lattice polytopes 
\cite{BN08}.

The fan $\Sigma$ in $N_\R$ can be constructed via the 
projection $$N_\R \oplus \R \to (N_\R \oplus \R)/ \R(n,r) \cong N_\R$$   
 of all proper faces of  the cone $C_{\Delta}^{\vee}$ 
along the  $1$-dimensional subspace generated by the unique interior lattice
point $(n, r)$ in the reflexive polytope $C_{\Delta}^{\vee}(r)$. 

It is well-known \cite[Corollary 7.10]{BD96} that the stringy Euler number 
of a generic Calabi-Yau hypersurface $Z$ in a Gorenstein toric 
Fano variety $X$ can be computed via volumes of faces $\theta \preceq \Delta$ and 
$\theta^*\preceq \Delta^*$ of $d$-dimensional 
reflexive 
polytopes $\Delta$ respectively $\Delta^*$ 
as 
\begin{align*}
e_{str}\left( Z \right) = c_{d-1}^{str} \left( Z \right)
= \sum_{k=0}^{d-3}  (-1)^{k} \sum_{\theta \preceq \Delta \atop \dim\left(\theta\right)=k+1}v(\theta) \cdot v\left(\theta^*\right). 
\end{align*} 

Using our previous results and the duality between 
faces $\theta \preceq \Delta$ and 
$\theta^*\preceq \Delta^*$ of Gorenstein 
polytopes $\Delta$ respectively $\Delta^*$, we can
generalize this combinatorial formula  
to the case of generic Calabi-Yau complete intersections
in Gorenstein toric Fano varieties. 

\begin{theo} \label{ecomref}
Let $X$ be a Gorenstein toric Fano variety 
associated with a $d$-di\-men\-sio\-nal 
Gorenstein polytope 
$\Delta \subseteq M_{\R}$ of index $r$ and 
$D$ an ample torus-invariant 
Cartier divisor on $X$ such that 
$\left[D\right] = \frac{1}{r}c_1(X)$.
Denote by $Z_1, \ldots, Z_r \subseteq X$
generic semiample Cartier divisors such that 
$\left[Z_1\right]= \ldots= \left[Z_{r}\right] = \left[D\right]= \frac{1}{r}c_1(X)$. 
Then the stringy Euler number of the Calabi-Yau complete intersection 
$S \defeq Z_1\cap \ldots \cap Z_{r}$ is 
\begin{align*}
c_{d-r}^{str}(S) = 
\sum_{k=0}^{d-r-1}
(-1)^k  \binom{k  +  r  -  1 }{r  - 1} \! \! \! \! 
\sum_{\theta \preceq \Delta \atop \dim(\theta)=k+r} 
v(\theta)  \cdot v\left(\theta^*\right) 
+ (-1)^{d-r}  \binom{ d - 1 }{r - 1} v(\Delta). 
\end{align*}
\end{theo}

\begin{proof} 
By Corollary \ref{ecomb2},  we have 
\begin{align*}
c_{d-r}^{str}\left( S\right) 
&= \sum_{k=0}^{d-r} (-1)^k \binom{k+r-1}{r-1} \sum_{\sigma \in \Sigma(d-r-k)}  v(\sigma) 
\cdot v \left(\Delta_D^{\sigma}\right),
\end{align*}
where $\Sigma$ is the associated fan to $X$ and
$\Delta_D^{\sigma}$ a face
of the lattice polytope $\Delta_D$. 
Let $\sigma$ be a  $(d-r-k)$-dimensional cone 
of $\Sigma$ ($0 \leq k \leq d-r-1$). Then 
$\sigma$ can be considered as a cone over a 
$(d-r-k-1)$-dimensional proper face of the reflexive 
polytope $(r\Delta)^*$, which we naturally identify with 
the corresponding proper face  $\theta^* \preceq \Delta^*$ of 
the dual Gorenstein polytope $\Delta^*$ \cite[Proposition 1.16]{BN08}. 
Therefore, we obtain
$v(\sigma) = v(\theta^*)$. On the other hand, the lattice polytope 
$\Delta_D$ is exactly the Gorenstein polytope $\Delta$ and 
 $\theta^*$ is the dual face to a $(k+r)$-dimensional face 
$\Delta_D^{\sigma} = \theta$ of $\Delta$. 
This implies 
\begin{align*}
\sum_{\sigma \in \Sigma(d-r-k)}  v(\sigma) 
\cdot v \left(\Delta_D^{\sigma}\right)
=  \sum_{\theta^* \preceq \Delta^* \atop \dim\left(\theta^*\right)=d-k-r-1}  
v\left(\theta^*\right) \cdot v(\theta)
=  \sum_{\theta \preceq \Delta \atop \dim\left(\theta \right)=k+r}  
v\left(\theta \right) \cdot v\left(\theta^*\right)
\end{align*}
for all $0 \leq k \leq d-r-1$. 
It remains to note that in the case  $k=d-r$ (i.e., $\dim (\sigma)=0$),
one has $\Delta_D^{\sigma} = \Delta$, 
$v \left(\Delta_D^{\sigma}\right) = v(\Delta)$, and $v(\sigma) =1$. 
\end{proof}

The combinatorial formula from Corollary \ref{324}
\begin{align*}
24 = \sum_{\theta \preceq \Delta \atop \dim(\theta)=1} v(\theta) \cdot v\left(\theta^* \right)
\end{align*}
for a $3$-dimensional reflexive polytope $\Delta$ can be 
generalized for arbitrary $d$-dimensional Gorenstein 
polytopes $\Delta$ $(d \geq 3)$ of index $r= d-2$. 

\begin{prop} \label{24delta}
Let $\Delta \subseteq M_{\R}$ be a $d$-dimensional Gorenstein polytope of index $r = d-2$. 
Then
\[24 = \sum_{\theta \preceq \Delta \atop \dim(\theta)=r} v(\theta) \cdot v\left(\theta^* \right)  
+ \frac{r(1-r)}{2} v(\Delta). \]
\end{prop}

\begin{proof}
Let  $S \defeq  Z_1 \cap \ldots \cap Z_r$ be a  generic 
Calabi-Yau complete intersection in the  Gorenstein toric Fano variety 
$X$ associated to $\Delta$, i.e.,  
$Z_1,\ldots,Z_r$ 
are $r$ generic ample Cartier divisors on $X$ such that
$\left[Z_1\right] = \ldots = \left[Z_r\right] = 
\frac{1}{r}c_1\left({X}\right)$. Then $\dim (S) =2$, i.e., 
$S$ is a  (possibly singular) $K3$-surface. The stringy Euler
number $c_2^{str}(S)$ of $S$ equals the usual Euler number $c_2(\widetilde{S})$
of the minimal (crepant) desingularization $\widetilde{S}$ of $S$. 
Since $\widetilde{S}$
is  a smooth $K3$-surface, we have $c_2^{str}(S) =c_2(\widetilde{S})=24$.  
Using Theorem \ref{ecomref}, we obtain 
\begin{align*}
24 = c_{2}^{str}\left(S \right)  = \sum_{\theta \preceq \Delta \atop \dim(\theta)=r}v(\theta) \cdot v\left(\theta^*\right)
-r \cdot \sum_{\theta \preceq \Delta \atop \dim(\theta)=d-1}v(\theta) \cdot v\left(\theta^*\right)
+\frac{(r+1)r}{2} v(\Delta).
\end{align*}
Since $r\Delta$ is a reflexive polytope, one has 
\[  r^d v(\Delta) =   v(r\Delta) =  \sum_{r\theta \preceq r\Delta 
\atop \dim(r\theta)=d-1} v(r\theta) = r^{d-1} 
\sum_{\theta \preceq \Delta \atop \dim(\theta)=d-1} v(\theta). \]
Moreover, $v(\theta^*)=1$ if $\dim(\theta^*) =0$.  
It remains to apply the equalities 
\begin{align*}
\sum_{\theta \preceq \Delta \atop \dim(\theta)=d-1}
v(\theta) \cdot v\left(\theta^*\right)
=\sum_{\theta \preceq \Delta \atop \dim(\theta)=d-1}
v(\theta) 
=r v(\Delta).
\end{align*}
\end{proof}

It was proved in \cite[Proposition 3.4]{BJ10} that 
the combinatorial identity 
\begin{align*}
12 = \sum_{\theta \preceq \Delta \atop \dim(\theta)=1} v(\theta) 
\cdot v\left(\theta^* \right)
\end{align*}
holds for any $3$-dimensional Gorenstein polytope $\Delta$  
of index $2$. We show that  this identity  can be 
generalized for arbitrary $d$-dimensional Gorenstein 
polytopes $\Delta$ $(d \geq 3)$  of index $r= d-1$.

\begin{prop}
Let $\Delta \subseteq M_{\R}$ be a $d$-dimensional Gorenstein polytope of index $r = d-1$. 
Then
\[12 = \sum_{\theta \preceq \Delta \atop \dim(\theta)=r-1} v(\theta) \cdot v\left(\theta^* \right) + \frac{r(1-r)+2}{2} v\left(\Delta \right). \]
\end{prop}

\begin{proof}
Let  $S \defeq  Z_1 \cap \ldots \cap Z_{r-1}$ be a  generic 
complete intersection in the  Gorenstein toric Fano variety 
$X$ associated to $\Delta$, i.e.,  
$Z_1,\ldots,Z_{r-1}$ 
are $r-1$ generic ample Cartier divisors on $X$ such that
$[D ] \defeq \left[Z_1\right] = \ldots = \left[Z_{r-1}\right] = 
\frac{1}{r}c_1\left({X}\right)$. 
Then $\dim (S) =2$, i.e., 
$S$ is a  (possibly singular) del Pezzo surface. The stringy Euler
number $c_2^{str}(S)$ of $S$ equals the usual Euler number $c_2(\widetilde{S})$
of the minimal (crepant) desingularization $\widetilde{S}$ of $S$.
 Then $\widetilde{S}$
is a smooth rational surface and  we have $c_2^{str}(S) =c_2(\widetilde{S})$ 
respectively $c_1(S)^2 =c_1(\widetilde{S})^2$. Using Noether's Theorem for 
$\widetilde{S}$, we obtain 
\begin{align*}
12 = c_2(\widetilde{S}) + c_1(\widetilde{S})^2  
= c_2^{str} \left( S \right)+c_1 \left( S \right)^2.  
\end{align*}
It remains to derive combinatorial formulas for $c_1 \left( S \right)^2$ and 
$ c_2^{str} \left( S \right)$.
By the adjunction formula, the anticanonical class of $S$ is the 
restriction of $c_1(X) - (r-1)[D]= [D]$  to $S$. Therefore, 
$c_1 \left( S \right)^2 =[D]^2 . [D]^{r-1} = [D]^d = v(\Delta)$.   
By  Corollary \ref{ecomb2}, we have 
\begin{align*}
c_{2}^{str}\left(S \right) 
&= \sum_{\sigma \in \Sigma(2) }v(\sigma) \cdot v\left(\Delta_D^{\sigma}\right)
-(r-1) \cdot \sum_{\sigma \in \Sigma(1) }v(\sigma) \cdot v\left(\Delta_D^{\sigma}\right)
+\frac{r(r-1)}{2} v(\Delta_D). 
\end{align*}
Using the same arguments  as in the 
proof of Theorem \ref{ecomref}, we can rewrite the above equation as 
\begin{align*}
c_{2}^{str}\left(S \right) &= \sum_{\theta \preceq \Delta \atop \dim(\theta)=d-2}v(\theta) \cdot v\left(\theta^*\right)
-(r-1) \cdot \sum_{\theta \preceq \Delta \atop \dim(\theta)=d-1}v(\theta) \cdot v\left(\theta^*\right)
+\frac{r(r-1)}{2} v(\Delta). 
\end{align*}
Using the last equation 
\begin{align*}
\sum_{\theta \preceq \Delta \atop \dim(\theta)=d-1}
v(\theta) \cdot v\left(\theta^*\right)
=r v(\Delta) 
\end{align*}
in the 
proof of Proposition  \ref{24delta}, we obtain 
\begin{align*}
c_2^{str} \left( S \right) 
&= \sum_{\theta \preceq \Delta \atop \dim(\theta)=r-1}  v(\theta)   \cdot v\left(\theta^* \right)
+ \frac{r(1-r)}{2} v\left(\Delta \right).
\end{align*}
\end{proof}


\end{document}